\documentclass[oneside,english]{siamart190516}
\usepackage[T1]{fontenc}
\usepackage[latin9]{inputenc}
\usepackage{array}
\usepackage{units}
\usepackage{url}
\usepackage{multirow}
\usepackage{amsmath}
\usepackage{amssymb}
\usepackage{graphicx}
\PassOptionsToPackage{normalem}{ulem}
\usepackage{ulem}

\makeatletter

\providecommand{\tabularnewline}{\\}

\usepackage{times}
\usepackage{lipsum}
\usepackage{amsfonts}
\usepackage{graphicx}
\usepackage{epstopdf}
\usepackage{algorithmic}
\usepackage[symbol]{footmisc}
\ifpdf%
  \DeclareGraphicsExtensions{.eps,.pdf,.png,.jpg}
\else
  \DeclareGraphicsExtensions{.eps}
\fi

\usepackage{booktabs}
\def\vec#1{\boldsymbol{#1}}

\newsiamremark{remark}{Remark}

\makeatother

\usepackage{babel}

\begin{document}
\title{Approximate Generalized Inverses with Iterative Refinement \\
for $\epsilon$-Accurate Preconditioning of Singular Systems}
\author{Xiangmin Jiao\footnotemark[1]\ \footnotemark[2] \and Qiao Chen\footnotemark[1]}

\maketitle
\footnotetext[1]{Department of Applied Mathematics and Statistics and Institute for Advanced Computational Science, Stony Brook University, Stony Brook, NY 11794, USA.}
\footnotetext[2]{Corresponding author. Email: xiangmin.jiao@stonybrook.edu.}
\headers{$\epsilon$-Accurate Preconditioning of Singular Systems}{X. Jiao and Q. Chen}
\begin{abstract}
We introduce a new class of preconditioners to enable flexible GMRES
to find a least-squares solution, and potentially the pseudoinverse
solution, of large-scale sparse, asymmetric, singular, and potentially
inconsistent systems. We develop the preconditioners based on a new
observation that \emph{generalized inverses} (i.e., $\boldsymbol{A}^{g}\in\{\boldsymbol{G}\mid\boldsymbol{A}\boldsymbol{G}\boldsymbol{A}=\boldsymbol{A}\}$)
enable the preconditioned Krylov subspaces to converge in a single
step. We then compute an \emph{approximate generalized inverse} (\emph{AGI})
efficiently using a \emph{hybrid incomplete factorization} (\emph{HIF}),
which combines multilevel incomplete LU with rank-revealing QR on
its final Schur complement. We define the criteria of $\epsilon$-accuracy
and stability of AGI to guarantee the convergence of preconditioned
GMRES for consistent systems. For inconsistent systems, we fortify
HIF with iterative refinement to obtain \emph{HIFIR}, which allows
accurate computations of the null-space vectors. By combining the
two techniques, we then obtain a new solver, called \emph{PIPIT},
for obtaining the pseudoinverse solutions for systems with low-dimensional
null spaces. We demonstrate the robustness of HIF and HIFIR and show
that they improve both accuracy and efficiency of the prior state
of the art by orders of magnitude for systems with up to a million
unknowns.
\end{abstract}

\section{Introduction\label{sec:Introduction}}

We consider the problem of finding an accurate and stable solution
of a potentially inconsistent linear system, 
\begin{equation}
\boldsymbol{A}\boldsymbol{x}\approx\boldsymbol{b},\label{eq:least-squares-linear-system}
\end{equation}
or more precisely 
\begin{equation}
\boldsymbol{x}_{\text{LM}}=\arg\min_{\boldsymbol{x}}\left\Vert \boldsymbol{b}-\boldsymbol{A}\boldsymbol{x}\right\Vert _{2},\label{eq:least-sq-solution}
\end{equation}
where $\boldsymbol{A}\in\mathbb{R}^{n\times n}\backslash\{\boldsymbol{0}\}$
is asymmetric and potentially singular, $\boldsymbol{b}\in\mathbb{R}^{n}$,
and $\boldsymbol{x}_{\text{LM}}\in\mathbb{R}^{n}$ is a least-squares
(LS) solution. In general, \eqref{eq:least-squares-linear-system}
is \emph{inconsistent} in that $\left\Vert \boldsymbol{b}-\boldsymbol{A}\boldsymbol{A}^{+}\boldsymbol{b}\right\Vert \gg\epsilon_{\text{mach}}\left\Vert \boldsymbol{b}\right\Vert $,
where $\boldsymbol{A}^{+}$ denotes the Moore--Penrose pseudoinverse
of $\boldsymbol{A}$ and $\epsilon_{\text{mach}}$ denotes the machine
epsilon of a given floating-point number system. Furthermore, $\boldsymbol{A}$
may be \emph{structurally singular}, in that its corresponding bipartite
graph may not have a full match between its rows and columns. Mathematically,
\eqref{eq:least-squares-linear-system} can be posed as a \emph{rank-deficient
least squares }(RDLS) problem, for which the \emph{pseudoinverse solution}
minimizes the 2-norm among the LS solutions, i.e., 
\begin{equation}
\boldsymbol{x}_{\mathrm{PI}}=\arg\min_{\boldsymbol{x}}\left\Vert \boldsymbol{x}\right\Vert _{2}\quad\text{subject to}\quad\min_{\boldsymbol{x}}\left\Vert \boldsymbol{b}-\boldsymbol{A}\boldsymbol{x}\right\Vert _{2}.\label{eq:pseudo-inverse-solution}
\end{equation}
For relatively small or moderate-sized systems, the pseudoinverse
solution can be obtained using the truncated singular value decomposition
(TSVD) \cite[p. 291]{Golub13MC}, i.e., $\boldsymbol{A}=\boldsymbol{U}\boldsymbol{\Sigma}\boldsymbol{V}^{T}$
, so that 
\begin{equation}
\boldsymbol{x}_{\mathrm{PI}}=\sum_{\{i\mid\sigma_{i}\gg\epsilon_{\text{mach}}\sigma_{1}\}}\left(\boldsymbol{u}_{i}^{T}\boldsymbol{b}/\sigma_{i}\right)\boldsymbol{v}_{i}+\mathcal{O}(\boldsymbol{\epsilon}_{\text{mach}}).\label{eq:svd-solution}
\end{equation}
Alternatively, an LS solution $\boldsymbol{x}_{\mathrm{LS}}$ can
be obtained using truncated QR with column pivoting \cite{Golub13MC}
(aka rank-revealing QR (RRQR) \cite{chan1987rank}), i.e., $\boldsymbol{A}\boldsymbol{P}=\boldsymbol{Q}\boldsymbol{R}$,
where the diagonal entries of $\boldsymbol{R}$ are nonnegative and
are in descending order. Then,
\begin{equation}
\boldsymbol{x}_{\mathrm{LS}}=\boldsymbol{P}_{1:r}\boldsymbol{R}_{1:r,1:r}^{-1}\left(\boldsymbol{Q}_{1:r}\right)^{T}\boldsymbol{b}+\mathcal{O}(\boldsymbol{\epsilon}_{\text{mach}}),\label{eq:least-squares-qrcp}
\end{equation}
where $r$ denotes the (numerical) rank of $\boldsymbol{R}$. If the
(right) null space of $\boldsymbol{A}$ (denoted by $\mathcal{N}(\boldsymbol{A}$))
is known \emph{a priori}, then $\boldsymbol{x}_{\mathrm{LS}}$ can
be converted to $\boldsymbol{x}_{\mathrm{PI}}$ by projecting off
its component in $\mathcal{N}(\boldsymbol{A}$).

In this work, we focus on effective preconditioners for Krylov subspace
(KSP) type methods for finding an LS solution, and potentially the
pseudoinverse solution, of large-scale, sparse, asymmetric singular
systems. Such systems often arise from partial differential equations
(PDEs) and other applications. $\boldsymbol{A}$ is often \emph{range
asymmetric}, i.e., $\mathcal{R}(\boldsymbol{A})\neq\mathcal{R}(\boldsymbol{A}^{T})$,
which is much more challenging than range-symmetric systems \cite{brown1997gmres}.
A direct solver (such as SuiteSparseQR \cite{davis2011algorithm})
or sparse SVD \cite{baglama2005augmented} is prohibitively expensive
computationally for large-scale systems with millions of unknowns.
Iterative methods (such as LSQR \cite{paige1982lsqr}, LSMR \cite{fong2011lsmr},
and variants of GMRES \cite{hayami2010gmres,morikuni2013inner,morikuni2015convergence})
and their preconditioned counterparts (such as with RIF \cite{benzi2003robust_positive_definite,benzi2003robust}
or incomplete QR \cite{li2006miqr,saad1988preconditioning,jennings1984incomplete})
are representative of the state of the art; however, significant challenges
remain open in terms of robustness and efficiency \cite{gould2017state}.

To develop effective preconditioners, we first generalize the theory
of optimal preconditioners for nonsingular systems to singular systems.
In particular, we show that a \emph{generalized inverse} of $\boldsymbol{A}$
(i.e., $\boldsymbol{A}^{g}\in\{\boldsymbol{G}\mid\boldsymbol{A}\boldsymbol{G}\boldsymbol{A}=\boldsymbol{A}\}$,
aka $\{1\}$-inverse) is an \emph{optimal} preconditioner in that
it enables a right-preconditioned KSP to converge to a (weighted)
LS solution in one iteration. Based on this optimality condition,
we propose a \emph{hybrid incomplete factorization} (\emph{HIF}) preconditioner,
which combines a multilevel incomplete LU (MLILU) factorization with
an RRQR factorization on the final Schur complement. We show that
HIF guarantees the convergence of preconditioned GMRES for consistent
systems with sufficiently small dropping thresholds. For inconsistent
systems, we fortify HIF with iterative refinement to obtain \emph{HIFIR
}(pronounced hi-fur) as a \emph{variable preconditioner} for flexible
GMRES (FGMRES) \cite{saad1993flexible}. HIFIR provides a memory-efficient
approach to improve the accuracy of HIF. We demonstrate the robustness
and efficiency of HIF and HIFIR on three important subclasses of singular
systems: (nearly) consistent systems, computation of null-space vectors,
and the pseudoinverse solutions of systems with a low-dimensional
null space (such as those from numerical PDEs). The theory and algorithms
described in this paper are extensible to complex matrices (by replacing
transpose with conjugate transpose) and to $m\times n$ RDLS problems
(by padding $\left|m-n\right|$ zero rows or columns for $m<n$ and
$m>n$, respectively). For simplicity of presentation, we focus on
$n\times n$ real matrices. However, all the results generalize to
complex matrices, and the implementation of HIFIR for both real and
complex matrices is available at \url{https://github.com/hifirworks/hifir}.

The remainder of the paper is organized as follows. In section~\ref{sec:Related-works},
we review some theoretical analyses of GMRES and some state-of-the-art
methods for RDLS systems. In section~\ref{sec:Optimal-preconditioning},
we derive the optimality conditions for right-preconditioning KSP
methods. In section~\ref{sec:HIF}, we introduce HIF for preconditioning
GMRES and prove its convergence for consistent asymmetric systems.
In section~\ref{sec:HIFIR}, we introduce HIFIR as a variable preconditioner
for FGMRES. In section~\ref{sec:Application-to-null-space}, we describe
the applications of HIFIR to computing the null spaces of matrices
and computing the pseudoinverse solution of inconsistent systems from
PDEs. In section~\ref{sec:Numerical-Experiments-and}, we present
numerical results and comparisons with some other techniques for systems
with up to a million unknowns. Section~\ref{sec:Conclusions} concludes
the paper with some discussions.

\section{\label{sec:Related-works}Flexible KSP and related methods}

We define some notation of KSP and its flexible variants, summarize
some previous theoretical analyses of preconditioned KSP, and review
some existing methods for RDLS systems.

\subsection{\label{subsec:Flexible-Krylov-subspaces} Flexible Krylov subspaces
with variable preconditioning}

Given a matrix $\boldsymbol{A}\in\mathbb{R}^{n\times n}$ and a vector
$\boldsymbol{v}\in\mathbb{R}^{n}$, the $k$th KSP associated with
$\boldsymbol{A}$ and $\boldsymbol{v}$ is 
\begin{equation}
\mathcal{K}_{k}(\boldsymbol{A},\boldsymbol{v})=\text{span}\{\boldsymbol{v},\boldsymbol{A}\boldsymbol{v},\dots,\boldsymbol{A}^{k-1}\boldsymbol{v}\}.\label{eq:Krylov-subspace}
\end{equation}
To introduce preconditioning, let us first consider the simpler case
where $\boldsymbol{A}$ is nonsingular. A right-preconditioned KSP
method with a fixed nonsingular preconditioner $\boldsymbol{M}$ solves
\begin{equation}
\boldsymbol{A}\boldsymbol{M}^{-1}\boldsymbol{y}=\boldsymbol{b}\label{eq:right-preconditioned-system}
\end{equation}
by finding a series of LS solutions $\boldsymbol{y}_{k}$ in $\mathcal{K}_{k}(\boldsymbol{A}\boldsymbol{M}^{-1},\boldsymbol{v})$,
where $\boldsymbol{v}$ is typically $\boldsymbol{b}$. The approximate
solution to the original equation is then $\boldsymbol{x}_{k}=\boldsymbol{M}^{-1}\boldsymbol{y}_{k}$.
Ideally, $\boldsymbol{M}^{-1}$ should be an accurate and stable approximation
to $\boldsymbol{A}^{-1}$, so that $\boldsymbol{A}\boldsymbol{M}^{-1}\approx\boldsymbol{A}\boldsymbol{A}^{-1}=\boldsymbol{I}$.
More generally, $\boldsymbol{A}$ and $\boldsymbol{M}$ may be singular.
Let $\boldsymbol{G}$ denote a \emph{right preconditioning operator}
(\emph{RPO}) such that $\boldsymbol{G}\boldsymbol{y}$ generalizes
$\boldsymbol{M}^{-1}\boldsymbol{y}$. We then solve
\begin{equation}
\boldsymbol{A}\boldsymbol{G}\boldsymbol{y}\approx\boldsymbol{b},\label{eq:right-preconditioned-RDLS}
\end{equation}
or more precisely $\min_{\boldsymbol{y}}\Vert\boldsymbol{A}\boldsymbol{G}\boldsymbol{y}-\boldsymbol{b}\Vert.$
The ideal choice of $\boldsymbol{G}$ for singular systems is less
apparent than that for nonsingular systems, and we will address this
issue in section~\ref{sec:Optimal-preconditioning}. In general,
the RPO may not be a matrix (e.g., multigrid preconditioners \cite{briggs2000multigrid,TOS00Multigrid})
and may vary from iteration to iteration \cite{saad1993flexible}.
For generality, let $\mathcal{G}_{k}$ denote the $k$th RPO.
\begin{definition}
\label{def:FKSP} Given a matrix $\boldsymbol{A}\in\mathbb{R}^{n\times n}$,
an initial vector $\boldsymbol{v}\in\mathbb{R}^{n}$, and variable
preconditioners $\boldsymbol{\mathcal{G}}_{k-1}=[\mathcal{G}_{1},\mathcal{G}_{2},\dots,\mathcal{G}_{k-1}]$,
the $k$th \emph{flexible Krylov subspace }(\emph{FKSP}) associated
with $\boldsymbol{A}$, $\boldsymbol{v}$, and $\boldsymbol{\mathcal{G}}_{k-1}$
is 
\begin{equation}
\mathcal{K}_{k}(\boldsymbol{A},\boldsymbol{v},\boldsymbol{\mathcal{G}}_{k-1})=\text{span}\{\boldsymbol{v},\boldsymbol{A}\mathcal{\mathcal{G}}_{1}(\boldsymbol{v}_{1}),\dots,\boldsymbol{A}\mathcal{\mathcal{G}}_{k-1}(\boldsymbol{v}_{k-1})\},\label{eq:FKSP}
\end{equation}
where $\boldsymbol{v}_{k-1}\in\mathcal{K}_{k-1}(\boldsymbol{A},\boldsymbol{v},\boldsymbol{\mathcal{G}}_{k-2})\backslash\{\boldsymbol{0}\}$
and $\boldsymbol{v}_{k-1}\perp\mathcal{K}_{k-2}(\boldsymbol{A},\boldsymbol{v},\boldsymbol{\mathcal{G}}_{k-3})$
for $k\geq2$. The \emph{flexible Krylov matrix, denoted by} $\boldsymbol{K}_{k}$,
is composed of the vectors in the right-hand side of \eqref{eq:FKSP}. 
\end{definition}

In Definition~\ref{def:FKSP}, we used a convention that $\mathcal{K}_{0}=\{\boldsymbol{0}\}$
so that $\boldsymbol{v}_{1}=\alpha\boldsymbol{v}$ for some $\alpha\neq0$.
The FKSP is said to have suffered a \emph{breakdown} if the vectors
in $\boldsymbol{K}_{k}$ are no longer linearly independent. In the
special case of right-preconditioned KSP with a fixed preconditioner
$\boldsymbol{M}$, we have $\boldsymbol{\mathcal{G}}_{k-1}=[\boldsymbol{M}^{-1},\dots,\boldsymbol{M}^{-1}]$.
Suppose $\boldsymbol{K}_{k}$ has full rank and the initial solution
is $\boldsymbol{0}$. Let $\boldsymbol{Q}_{k}$ denote the orthonormal
basis of $\mathcal{K}{}_{k}$ obtained from the QR factorization of
$\boldsymbol{K}_{k}$ through a (generalized) Arnoldi process \cite{saad1993flexible},
$\boldsymbol{q}_{j}$ be the $j$th column of $\boldsymbol{Q}_{k}$,
and $\boldsymbol{Z}_{k}$ be composed of $\boldsymbol{z}_{j}=\mathcal{\mathcal{G}}_{j}(\boldsymbol{q}_{j})$
for $j=1,2,\dots,k$. FGMRES \cite{saad1993flexible} finds $\boldsymbol{y}_{k}$,
so that $\boldsymbol{Z}_{k}\boldsymbol{y}_{k}$ is the LS solution
of \eqref{eq:least-squares-linear-system} projected onto $\mathcal{K}{}_{k+1}$,
and then $\boldsymbol{x}_{k}=\boldsymbol{Z}_{k}\boldsymbol{y}_{k}$
\cite[Section 9.4]{saad2003iterative}. Note that Definition~\ref{def:FKSP}
is specifically for FGMRES. For other flexible inner-outer KSP methods
\cite{simoncini2002flexible}, it requires adaptation before it can
be applied. It may also be generalized to left (or split) preconditioners.
However, left preconditioning may severely distort the residual and,
in turn, lead to early termination or false stagnation for ill-conditioned
systems \cite{ghai2017comparison}.

\subsection{\label{subsec:Convergence-of-right-preconditio}Breakdown conditions
of preconditioned and FKSP}

From the theoretical point of view, the convergence of KSP without
preconditioning or with fixed preconditioners is well understood \cite{brown1997gmres,calvetti2000gmres,hayami2010gmres,reichel2005bfgmres,saad2003iterative}.
In the following, we summarize two theorems on KSP with fixed preconditioners,
which are generalizations of some well-known results on unpreconditioned
KSP \cite{brown1997gmres,calvetti2000gmres}, as well as a generalization
of a theorem on FKSP \cite[Proposition 9.3]{saad2003iterative} to
singular systems. These results will provide the theoretical foundation
for the proposed preconditioner and solver in later sections.

For consistent systems, GMRES with a fixed preconditioner is guaranteed
to converge under some reasonable assumptions, as stated by the following
theorem.
\begin{theorem}
\label{thm:consistent-range-asymmetric} If \eqref{eq:least-squares-linear-system}
is consistent and $\mathcal{N}(\boldsymbol{A}\boldsymbol{G})\cap\mathcal{R}(\boldsymbol{A}\boldsymbol{G})=\{\boldsymbol{0}\}$,
then GMRES with RPO $\boldsymbol{G}$ does not break down until finding
an LS solution $\boldsymbol{x}_{\text{LS}}$ of \eqref{eq:least-squares-linear-system}
for all $\boldsymbol{b}\in\mathcal{R}(\boldsymbol{A})$ and $\boldsymbol{x}_{0}\in\mathbb{R}^{n}$
if and only if $\mathcal{R}(\boldsymbol{A})=\mathcal{R}(\boldsymbol{A}\boldsymbol{G})$.
Furthermore, $\boldsymbol{x}_{\text{LS}}$ is the pseudoinverse solution
of \eqref{eq:least-squares-linear-system} if $\mathcal{R}(\boldsymbol{G})=\mathcal{R}(\boldsymbol{A}^{T})$.
\end{theorem}

\begin{proof}
For the unpreconditioned GMRES, the theorem is equivalent to Theorem~2.6
in \cite{brown1997gmres}. For preconditioned GMRES, the breakdown-free
property follows from applying \cite[Theorem 2.6]{brown1997gmres}
to the preconditioned KSP. It is easy to show that at convergence,
$\boldsymbol{x}_{\text{LS}}=\boldsymbol{G}\boldsymbol{y}_{*}$ is
an LS solution if and only if $\mathcal{R}(\boldsymbol{A})=\mathcal{R}(\boldsymbol{A}\boldsymbol{G})$
\cite[Theorem 3.1]{hayami2010gmres}. Finally, if $\mathcal{R}(\boldsymbol{G})=\mathcal{R}(\boldsymbol{A}^{T})$,
then $\boldsymbol{x}_{\text{LS}}\perp\mathcal{N}(\boldsymbol{A})$,
so $\boldsymbol{x}_{\text{LS}}$ is the pseudoinverse solution.
\end{proof}
For unpreconditioned KSP, Theorem~\ref{thm:consistent-range-asymmetric}
is mathematically equivalent to an earlier result of Freund and Hochbruck
\cite[Theorem 2.6]{freund1994use}, as noted by Brown and Walker in
\cite{brown1997gmres}. 

\begin{remark}If $\mathcal{N}(\boldsymbol{A})\cap\mathcal{R}(\boldsymbol{A})=\{\boldsymbol{0}\}$,
$\boldsymbol{A}$ is said to have index 1. Recall that the \emph{index}
of $\boldsymbol{A}$, i.e., $\text{idx}(\boldsymbol{A})$, is the
smallest number $k$ such that $\mathcal{N}(\boldsymbol{A}^{k+1})=\mathcal{N}(\boldsymbol{A}^{k})$,
or the maximum dimension of the Jordan blocks associated with the
zero eigenvalue of $\boldsymbol{A}$. Note that for any matrix $\boldsymbol{A}$
with a defective eigenvalue $\lambda$, $\text{idx}(\boldsymbol{A}-\lambda\boldsymbol{I})>1$.
Hence, non-index-1 matrices are rare but not uncommon. Some methods,
such as DGMRES \cite{sidi2001dgmres}, were developed to compute the
so-called Drazin-inverse solution \cite[p. 356]{Golub13MC} for such
matrices. However, the index is in general not preserved under preconditioning,
so an effective preconditioner can convert a non-index-1 matrix to
an index-1 matrix as we will show in section~\ref{subsec:epsilon-accurate-AGI}.
Hence, we do not need other special treatments for high-index matrices.\end{remark}

For inconsistent systems, however, KSP with a fixed preconditioner
has some fundamental difficulties due to the following theorem. 
\begin{theorem}
\label{thm:GMRES-range-symmetric}GMRES with RPO $\boldsymbol{G}$
does not break down until finding an LS solution $\boldsymbol{x}_{\text{LS}}$
of \eqref{eq:least-squares-linear-system} for all $\boldsymbol{b}\in\mathbb{R}^{n}$
and $\boldsymbol{x}_{0}\in\mathbb{R}^{n}$ if and only if $\boldsymbol{A}\boldsymbol{G}$
is range symmetric and $\mathcal{R}(\boldsymbol{A})=\mathcal{R}(\boldsymbol{A}\boldsymbol{G})$.
Furthermore, $\boldsymbol{x}_{\text{LS}}$ is the pseudoinverse solution
if $\mathcal{R}(\boldsymbol{G})=\mathcal{R}(\boldsymbol{A}^{T})$.
\end{theorem}

\begin{proof}
For unpreconditioned GMRES, the theorem is equivalent to the well-known
result due to Brown and Walker \cite[Theorem 2.4]{brown1997gmres}.
With $\boldsymbol{G}$ as the right preconditioner, Theorem 2.4 of
\cite{brown1997gmres} implies that the preconditioned KSP finds an
LS solution $\boldsymbol{y}_{*}$ of the preconditioned system if
and only if $\boldsymbol{A}\boldsymbol{G}$ is range symmetric. $\boldsymbol{x}_{*}=\boldsymbol{G}\boldsymbol{y}_{*}$
is an LS solution of the original system if and only if $\mathcal{R}(\boldsymbol{A})=\mathcal{R}(\boldsymbol{A}\boldsymbol{G})$.
The last part follows the same argument as for Theorem~\ref{thm:consistent-range-asymmetric}.
\end{proof}
Theorem~\ref{thm:consistent-range-asymmetric} implies that a fixed
preconditioner for inconsistent systems must preserve the range of
$\boldsymbol{A}$ (i.e., $\mathcal{R}(\boldsymbol{A})=\mathcal{R}(\boldsymbol{A}\boldsymbol{G})$)
and at the same time ensure range symmetry of $\boldsymbol{A}\boldsymbol{G}$.
While preserving the range is relatively easy, it is difficult to
construct an RPO $\boldsymbol{G}$ for a singular $\boldsymbol{A}$
that satisfies both constraints simultaneously. One exception is $\boldsymbol{G}=\boldsymbol{A}^{T}$
as in \cite{gould2017state,hayami2010gmres,morikuni2015convergence},
but $\boldsymbol{G}=\boldsymbol{A}^{T}$ would square the condition
number and in turn slow down the convergence of the ``preconditioned''
system $\boldsymbol{A}\boldsymbol{A}^{T}\boldsymbol{y}=\boldsymbol{b}$;
see section~\ref{subsec:Review-of-related} for more discussions.
Hence, it is very challenging, if not impossible, to construct effective
fixed preconditioners that are robust for general inconsistent systems.
In practice, preconditioned GMRES often breaks down for some systems
before reaching the desired precision (e.g., near machine precision
\cite{reichel2005bfgmres}); see, e.g., Figures~\ref{fig:bcsstk35}
and \ref{fig:shyy161}. Theorem~\ref{thm:consistent-range-asymmetric}
can be viewed as a reinterpretation of the analysis of AB-GMRES in
\cite{hayami2010gmres}, where $\boldsymbol{B}$ is analogous to $\boldsymbol{G}$
in Theorem~\ref{thm:GMRES-range-symmetric}. Hayami, Yin, and Ito
\cite{hayami2010gmres} analyzed the convergence conditions on $\boldsymbol{B}$
based on $\boldsymbol{A}^{T}$. In this work, we derive the optimality
conditions for $\boldsymbol{G}$ based on generalized inverses of
$\boldsymbol{A}$.

The challenge associated with KSP with fixed preconditioners can be
mitigated by the FKSP with a variable preconditioner, due to the following
theorem.
\begin{theorem}
\label{thm:FGMRES-var-precond} If FGMRES with variable preconditioners
$\boldsymbol{\mathcal{G}}=[\mathcal{G}_{1},\mathcal{G}_{2},\dots]$
does not break down until step $k+1$, and if the projection of $\boldsymbol{b}$
onto $\mathcal{R}(\boldsymbol{A})$, denoted by $\text{proj}_{\mathcal{R}(\boldsymbol{A})}\boldsymbol{b}$,
is in $\sum_{i=1}^{k}\mathcal{R}(\boldsymbol{A}\mathcal{G}_{i})$,
then it finds an LS solution of \eqref{eq:least-squares-linear-system}
for an initial guess $\boldsymbol{x}_{0}\in\mathbb{R}^{n}$. Conversely,
if FGMRES finds an LS solution of \eqref{eq:least-squares-linear-system},
then $\sum_{i=1}^{k}\mathcal{R}(\boldsymbol{A}\mathcal{G}_{i})$ must
contain $\text{proj}_{\mathcal{R}(\boldsymbol{A})}\boldsymbol{b}$.
\end{theorem}

\begin{proof}
By assumption, FGMRES does not break down until step $k+1$, $\mathcal{K}_{k}(\boldsymbol{A},\boldsymbol{b},\boldsymbol{\mathcal{G}}_{k})=\mathcal{K}_{k+1}(\boldsymbol{A},\boldsymbol{b},\boldsymbol{\mathcal{G}}_{k})\subseteq\sum_{i=1}^{k}\mathcal{R}(\boldsymbol{A}\mathcal{G}_{i})+\{\boldsymbol{b}\}$.
Let $\boldsymbol{Q}_{k}\boldsymbol{R}_{k}$ denote the QR factorization
of $\boldsymbol{K}_{k}$, and let $\boldsymbol{Z}_{k}$ be composed
of $\boldsymbol{z}_{j}=\mathcal{\mathcal{G}}_{j}(\boldsymbol{q}_{j})$,
where $\boldsymbol{q}_{j}$ is the $j$th column of $\boldsymbol{Q}_{k}$
for $j=1,2,\dots,k$. If $\text{proj}_{\mathcal{R}(\boldsymbol{A})}\boldsymbol{b}\in\sum_{i=1}^{k}\mathcal{R}(\boldsymbol{A}\mathcal{G}_{i})$,
then there is an LS solution $\boldsymbol{y}_{k}\in\arg\min_{k}\left\Vert \boldsymbol{b}-\boldsymbol{A}\boldsymbol{Z}_{k}\boldsymbol{y}\right\Vert _{2}$,
so \textbf{$\boldsymbol{x}_{k}=\boldsymbol{Z}_{k}\boldsymbol{y}_{k}\in\arg\min_{\boldsymbol{x}}\left\Vert \boldsymbol{b}-\boldsymbol{A}\boldsymbol{x}\right\Vert _{2}$}.
Conversely, if $\text{proj}_{\mathcal{R}(\boldsymbol{A})}\boldsymbol{b}\not\in\sum_{i=1}^{k}\mathcal{R}(\boldsymbol{A}\mathcal{G}_{i})$,
then $\boldsymbol{x}_{k}=\boldsymbol{Z}_{k}\boldsymbol{y}_{k}$ cannot
be an LS solution of $\min_{\boldsymbol{x}}\left\Vert \boldsymbol{b}-\boldsymbol{A}\boldsymbol{x}\right\Vert _{2}$.
\end{proof}
Theorem~\ref{thm:FGMRES-var-precond} generalizes Proposition 9.3
in \cite{saad2003iterative} on FGMRES for nonsingular systems. 
For unpreconditioned GMRES, a variant of Theorem~\ref{thm:FGMRES-var-precond}
was proven by Calvetti, Lewis, and Reichel \cite[Theorem 2.2]{calvetti2000gmres},
which assumed $k=\text{rank}(\boldsymbol{A})$. From a practical viewpoint,
Theorem~\ref{thm:FGMRES-var-precond} indicates a potential advantage
of variable preconditioning over fixed preconditioning: As long as
$\mathcal{R}(\boldsymbol{A}\mathcal{G}_{i})$ can continue introducing
a new component into the range space until $\sum_{i=1}^{k}\mathcal{R}(\boldsymbol{A}\mathcal{G}_{i})$
contains $\text{proj}_{\mathcal{R}(\boldsymbol{A})}\boldsymbol{b}$,
 then FGMRES would not break down, even if $\mathcal{R}(\boldsymbol{A}\mathcal{G}_{i})\neq\mathcal{R}(\boldsymbol{A})$
or $\boldsymbol{A}\mathcal{G}_{i}$ is not range symmetric in some
steps. The new components could be introduced, for example, via random
perturbations (similar to those in BFGMRES \cite{reichel2005bfgmres}),
or via inner iterations (similar to the original FGMRES for nonsingular
systems \cite{saad1993flexible} and its variants in AB-GMRES for
singular systems \cite{morikuni2015convergence}). In addition, restarts
with random perturbations or with different variable preconditioners
in the inner iterations can also increase the probability of introducing
new components. An alternative approach to overcome the potential
breakdown of GMRES for inconsistent systems is to convert inconsistent
systems into consistent ones \cite{brown1997gmres}. In this work,
we take the latter approach, by introducing HIF with iterative refinement
as a variable preconditioner to compute the null-space vectors, as
we will describe in section~\ref{sec:Application-to-null-space}.

The preceding analysis assumed exact arithmetic. It is well known
that with floating-point arithmetic, the asymptotic convergence rate
of GMRES (and other KSP methods) for diagonalizable nonsingular systems
depends on the ratio of the extreme eigenvalues and the conditioning
of the eigenvectors of $\boldsymbol{A}$ \cite{saad2003iterative}.
If $\boldsymbol{A}$ is nearly symmetric and positive definite, the
number of iterations is $\mathcal{O}(\sqrt{\kappa(\boldsymbol{A})})$;
however, if $\boldsymbol{A}$ is indefinite, it may require up to
$\mathcal{O}(\kappa(\boldsymbol{A}))$ iterations \cite[Section 9.2]{benzi2005numerical}.
These analyses can be generalized to KSP with a fixed preconditioner
(see, e.g., \cite{morikuni2015convergence}), using the following
notion of the condition number.
\begin{definition}
\label{def:condition-number} Given a potentially singular matrix
$\boldsymbol{A}\in\mathbb{R}^{n\times n}$, the \emph{2-norm condition
number }of $\boldsymbol{A}$ is the ratio between the largest and
the smallest nonzero singular values of $\boldsymbol{A}$, i.e., $\kappa_{2}(\boldsymbol{A})=\sigma_{1}(\boldsymbol{A})/\sigma_{r}(\boldsymbol{A})$,
where $r=\text{\emph{rank}}(\boldsymbol{A})$.
\end{definition}

Note that 
\begin{equation}
\sigma_{r}(\boldsymbol{A})=\min_{\boldsymbol{u}\in\mathcal{R}(\boldsymbol{A})\backslash\{\boldsymbol{0}\}}\frac{\left\Vert \boldsymbol{u^{T}}\boldsymbol{A}\right\Vert }{\left\Vert \boldsymbol{u}\right\Vert }=\min_{\boldsymbol{v}\in\mathcal{R}(\boldsymbol{A}^{T})\backslash\{\boldsymbol{0}\}}\frac{\left\Vert \boldsymbol{A}\boldsymbol{v}\right\Vert }{\left\Vert \boldsymbol{v}\right\Vert }.\label{eq:rth-singular-value}
\end{equation}
The asymptotic convergence analysis for GMRES assumes unlimited dimensions
of the KSP. Due to limited memory resources, restarted GMRES (or FGMRES)
is typically used for large-scale systems. Restart can lead to complications,
such as stagnation \cite{saad2003iterative}. Our goal in this work
is to develop near-optimal preconditioners, so that restarted FGMRES
can converge quickly, ideally before restart is invoked, even for
systems with millions of unknowns.

\subsection{\label{subsec:Review-of-related}Review of related iterative methods
and preconditioners}

Iterative methods and preconditioners for RDLS systems have received
significant attention in recent years \cite{Choi11MINRES,fong2011lsmr,gould2017state,hayami2010gmres,morikuni2018gmres}.
We review a few related methods in increasing generality. We will
also briefly analyze them based on the theory summarized in section~\ref{subsec:Convergence-of-right-preconditio}.

\subsubsection{\label{subsec:Symmetric-singular-systems}Symmetric singular systems}

MINRES \cite{Paige75MINRES} can solve symmetric indefinite systems,
including singular but consistent systems. For inconsistent symmetric
systems, Choi, Paige, and Saunders introduced MINRES-QLP \cite{Choi11MINRES},
which incrementally estimates the extreme singular values and corresponding
singular vectors of the Hessenberg matrix to detect ill-conditioning
and filter out the null-space component. Preconditioning for MINRES-type
methods is subtle because symmetric preconditioning alters the range
in general, leading to a different weighted-least-squares (WLS) solution
for inconsistent systems. Although symmetric preconditioning may be
posed as right preconditioning, such as in \cite[Theorem 1]{sugihara2020right},
a true LS solution still cannot be obtained for inconsistent systems
\cite[Proposition 1]{sugihara2020right}.

\subsubsection{\label{subsec:Asymmetric-singular-systems}Asymmetric singular systems}

For nearly singular systems, Brown and Walker \cite{brown1997gmres}
proposed to terminate GMRES when ill-conditioning is detected by monitoring
the condition number of the Krylov matrix $\boldsymbol{K}$. By taking
advantage of its QR factorization, the condition number can be estimated
efficiently in an incremental fashion \cite{brown1997gmres}. However,
this strategy is not robust enough to achieve near machine precision
(see \cite[Example 4.1]{reichel2005bfgmres}). In \cite{calvetti2000gmres},
Calvetti, Lewis, and Reichel proposed to improve the conditioning
of the KSP by using $\mathcal{K}(\boldsymbol{A},\boldsymbol{A}\boldsymbol{b})$
in place of $\mathcal{K}(\boldsymbol{A},\boldsymbol{b})$, and they
referred to the approach as RRGMRES. For range-symmetric systems,
RRGMRES is guaranteed to converge to the pseudoinverse solution. For
range-asymmetry systems, however, RRGMRES does not appear to have
significant advantages over GMRES \cite{morikuni2018gmres}. In \cite{reichel2005bfgmres},
Reichel and Ye introduced BFGMRES to alleviate potential breakdowns
of GMRES by randomly perturbing the KSP when ill-conditioning is detected.

\subsubsection{\label{subsec:General-rank-deficient-least}General rank-deficient
least-squares systems}

The analysis in section~\ref{subsec:Convergence-of-right-preconditio}
applies to structurally singular systems, so they also apply to $m\times n$
RDLS systems. However, an iterative solver may potentially take advantage
of $m<n$ or $m>n$ for the so-called underdetermined or overdetermined
cases, respectively, especially if $\boldsymbol{A}$ has full rank.
Examples include AB- and BA-GMRES in \cite{hayami2010gmres}, which
established a general framework with fixed preconditioners. Hayami
et al. \cite{hayami2010gmres} developed a preconditioner based on
robust incomplete factorization (RIF) of Benzi and T\r{u}ma \cite{benzi2003robust_positive_definite,benzi2003robust}.
They showed that RIF-preconditioned BA- and AB-GMRES significantly
outperformed preconditioned LSQR for some full-rank systems. However,
RIF cannot guarantee range symmetry of the preconditioned system,
so Theorem~\ref{thm:GMRES-range-symmetric} implies that RIF-preconditioned
AB- or BA-GMRES cannot be robust for general systems if high precision
of the solution is required.

To achieve robustness in AB- and BA-GMRES, it appears that the most
practical choice for $\boldsymbol{B}$ is $\boldsymbol{A}^{T}$ \cite{gould2017state,morikuni2015convergence}.
 With such a choice, the KSPs in AB- and BA-GMRES without preconditioning
are essentially $\mathcal{K}_{k}(\boldsymbol{A}\boldsymbol{A}^{T},\boldsymbol{b})$
and $\mathcal{K}_{k}(\boldsymbol{A}^{T}\boldsymbol{A},\boldsymbol{A}^{T}\boldsymbol{b})$,
respectively. The latter KSP is the same as those used in CGLS \cite{bjorck1996numerical}
(also known as CGNR \cite{saad2003iterative} or CGN \cite{trefethen1997numerical}),
LSQR \cite{paige1982lsqr}, and LSMR \cite{fong2011lsmr}. Mathematically,
CGLS and LSQR are equivalent to applying the conjugate gradient method
\cite{hestenes1952methods} to the normal equation, and LSMR is equivalent
to applying MINRES to the normal equation. These methods can all compute
the pseudoinverse solution. Unfortunately, $\boldsymbol{A}^{T}\boldsymbol{A}$
(or $\boldsymbol{A}\boldsymbol{A}^{T}$) squares the condition number
of $\boldsymbol{A}$ and, in turn, may square the number of iterations
in the worst case. 

The preconditioning of CGLS-type methods has attracted significant
attention. Example preconditioners include incomplete QR \cite{saad1988preconditioning,jennings1984incomplete},
RIF \cite{benzi2003robust,benzi2003robust_positive_definite}, and
inner iterations based on SSOR \cite{morikuni2013inner,morikuni2015convergence}.
 Like RIF and incomplete QR, our proposed HIF preconditioner is also
based on incomplete factorization. However, HIF preconditions $\boldsymbol{A}$
using an approximate generalized inverse (AGI) in FGMRES, which avoids
squaring the condition numbers in the first place. 

\section{\label{sec:Optimal-preconditioning}Optimality and near-optimality
conditions of preconditioners}

In this section, we analyze the conditions for right-preconditioned
GMRES to converge rapidly to an LS solution, ideally in just one iteration.
From section~\ref{subsec:Convergence-of-right-preconditio}, an intuitive
choice of such an optimal preconditioner is $\boldsymbol{G}=\boldsymbol{A}^{+}$,
which enables GMRES to converge to a pseudoinverse solution in one
iteration. However, $\boldsymbol{A}^{+}$ is not the unique optimal
preconditioner. For example, if $\boldsymbol{G}=\boldsymbol{P}_{1:r}\boldsymbol{R}_{1:r,1:r}^{-1}\left(\boldsymbol{Q}_{1:r}\right)^{T}$
as in \eqref{eq:least-squares-qrcp}, then $\boldsymbol{A}\boldsymbol{G}=\boldsymbol{A}\boldsymbol{A}^{+}$,
even though $\boldsymbol{G}\neq\boldsymbol{A}^{+}$. It is desirable
to identify the complete class of (near) optimal preconditioners.

\subsection{\label{subsec:Generalized-inverses}Preliminary: generalized inverses}

The pseudoinverse $\boldsymbol{A}^{+}$ of $\boldsymbol{A}\in\mathbb{R}^{m\times n}$,
introduced by Moore in 1920 \cite{moore1920reciprocal}, is a particular
case of generalized inverses of $\boldsymbol{A}$ \cite{rao1972generalized}.
\begin{definition}
\label{def:generalized-inverse}\cite[Definitions 2.2]{rao1972generalized}
Given a potentially rank-deficient $\boldsymbol{A}\in\mathbb{R}^{m\times n}$,
$\boldsymbol{A}^{g}$ is a \emph{generalized inverse} of $\boldsymbol{A}$
if and only if $\boldsymbol{A}\boldsymbol{A}^{g}\boldsymbol{A}=\boldsymbol{A}$.
\end{definition}

As an example, for $\boldsymbol{A}=\begin{bmatrix}\boldsymbol{I} & \boldsymbol{E}\\
 & \boldsymbol{0}
\end{bmatrix}$, $\boldsymbol{A}^{g}=\begin{bmatrix}\boldsymbol{I} & \boldsymbol{F}\\
 & \boldsymbol{C}
\end{bmatrix}$ is a generalized inverse of $\boldsymbol{A}$ for any $\boldsymbol{F}$
and $\boldsymbol{C}$ with compatible dimensions. $\boldsymbol{A}^{g}$
is also referred to as a $\{1\}$-inverse \cite{ben2003generalized}
in that it enforces only the first of the four Penrose conditions
\cite{penrose1955generalized}. An alternative definition of $\boldsymbol{A}^{g}$
(e.g., \cite[Definitions 2.3]{rao1972generalized}\footnote{The first part of Lemma~\ref{lem:pseudo-inv-range-preservation}
is equivalent to the first part of Definition 2.3 in \cite{rao1972generalized}.
However, the second part of \cite[Definition 2.3]{rao1972generalized}
should have read $\mathcal{N}(\boldsymbol{A}^{g}\boldsymbol{A})=\mathcal{N}(\boldsymbol{A})$
instead of $\mathcal{R}(\boldsymbol{A}^{g}\boldsymbol{A})=\mathcal{R}(\boldsymbol{A})$.
In general, $\mathcal{R}(\boldsymbol{A}^{g}\boldsymbol{A})\neq\mathcal{R}(\boldsymbol{A})$;
otherwise, $\mathcal{N}(\boldsymbol{A}^{g}\boldsymbol{A})=\mathcal{N}(\boldsymbol{A})$
would have implied $\mathcal{R}(\boldsymbol{A})\cap\mathcal{N}(\boldsymbol{A})=\{\boldsymbol{0}\}$.}) is given by the following lemma.
\begin{lemma}
\label{lem:pseudo-inv-range-preservation} $\boldsymbol{A}^{g}$ is
a generalized inverse of $\boldsymbol{A}$ if and only if $\boldsymbol{A}\boldsymbol{A}^{g}$
is idempotent (i.e., $\left(\boldsymbol{A}\boldsymbol{A}^{g}\right)^{2}=\boldsymbol{A}\boldsymbol{A}^{g}$)
and $\mathcal{R}(\boldsymbol{A}\boldsymbol{A}^{g})=\mathcal{R}(\boldsymbol{A})$,
or $\boldsymbol{A}^{g}\boldsymbol{A}$ is idempotent and $\mathcal{N}(\boldsymbol{A}^{g}\boldsymbol{A})=\mathcal{N}(\boldsymbol{A})$.
\end{lemma}

\begin{proof}
We prove only the second part, since it differs from \cite[Definitions 2.3]{rao1972generalized},
and it is easy to adapt the argument to prove the first part. If $\boldsymbol{A}\boldsymbol{A}^{g}\boldsymbol{A}=\boldsymbol{A}$,
then $\left(\boldsymbol{A}^{g}\boldsymbol{A}\right)^{2}=\boldsymbol{A}^{g}\left(\boldsymbol{A}\boldsymbol{A}^{g}\boldsymbol{A}\right)=\boldsymbol{A}^{g}\boldsymbol{A}$.
Furthermore, $\mathcal{N}(\boldsymbol{A})=\mathcal{N}(\boldsymbol{A}\boldsymbol{A}^{g}\boldsymbol{A})\supseteq\mathcal{\mathcal{N}}(\boldsymbol{A}^{g}\boldsymbol{A})\supseteq\mathcal{\mathcal{N}}(\boldsymbol{A})$,
so $\mathcal{\mathcal{N}}(\boldsymbol{A}^{g}\boldsymbol{A})=\mathcal{\mathcal{N}}(\boldsymbol{A})$.
Conversely, if $\boldsymbol{A}^{g}\boldsymbol{A}$ is idempotent,
$\boldsymbol{A}^{T}\left(\boldsymbol{A}^{g}\right)^{T}\boldsymbol{u}=\boldsymbol{u}$
for all $\boldsymbol{u}\in\mathcal{R}\left(\boldsymbol{A}^{T}\left(\boldsymbol{A}^{g}\right)^{T}\right)$.
Under the condition of $\mathcal{N}(\boldsymbol{A}^{g}\boldsymbol{A})=\mathcal{N}(\boldsymbol{A})$,
$\mathcal{R}(\boldsymbol{A}^{T}\left(\boldsymbol{A}^{g}\right)^{T})=\mathcal{R}(\boldsymbol{A}^{T})$.
Hence, $\boldsymbol{A}^{T}\left(\boldsymbol{A}^{g}\right)^{T}\boldsymbol{A}^{T}\boldsymbol{v}=\boldsymbol{A}^{T}\boldsymbol{v}$
for all $\boldsymbol{v}\in\mathbb{R}^{m}$ and, in turn, $\boldsymbol{A}\boldsymbol{A}^{g}\boldsymbol{A}=\boldsymbol{A}$.
\end{proof}
By definition, both $\boldsymbol{A}\boldsymbol{A}^{+}$ and $\boldsymbol{A}\boldsymbol{A}^{g}$
are projectors. However, $\boldsymbol{A}\boldsymbol{A}^{+}$ is an
orthogonal projector, but $\boldsymbol{A}\boldsymbol{A}^{g}$ is oblique
in general. Continuing the earlier example, $\boldsymbol{A}\boldsymbol{A}^{g}=\begin{bmatrix}\boldsymbol{I} & \boldsymbol{F}+\boldsymbol{E}\boldsymbol{C}\\
 & \boldsymbol{0}
\end{bmatrix}$, which is an oblique projector. The following properties will also
turn out to be useful.
\begin{lemma}
\label{lem:empty-intersect} If $\boldsymbol{A}^{g}$ is a generalized
inverse of $\boldsymbol{A}$, then $\mathcal{R}(\boldsymbol{A}^{g})\cap\mathcal{\mathcal{N}}(\boldsymbol{A})=\{\boldsymbol{0}\}$
and $\mathcal{R}(\boldsymbol{A})\cap\mathcal{\mathcal{N}}(\boldsymbol{A}^{g})=\{\boldsymbol{0}\}$.
\end{lemma}

\begin{proof}
If there were $\text{\ensuremath{\boldsymbol{v}}}\in\mathcal{R}(\boldsymbol{A}^{g})\cap\mathcal{\mathcal{N}}(\boldsymbol{A})\backslash\{\boldsymbol{0}\}$,
then $\text{rank}(\boldsymbol{A}\boldsymbol{A}^{g})<\text{rank}(\boldsymbol{A})$,
which contradicts $\mathcal{R}(\boldsymbol{A}\boldsymbol{A}^{g})=\mathcal{R}(\boldsymbol{A})$.
The proof for $\mathcal{R}(\boldsymbol{A})\cap\mathcal{\mathcal{N}}(\boldsymbol{A}^{g})=\{\boldsymbol{0}\}$
is similar.
\end{proof}
\begin{proposition}
\label{prop:diagnalizability} If $\boldsymbol{A}^{g}$ is a generalized
inverse of $\boldsymbol{A}\in\mathbb{R}^{m\times n}$, then $\boldsymbol{A}\boldsymbol{A}^{g}$
is diagonalizable, and its eigenvalues are all zeros and ones. In
other words, there exists a nonsingular matrix $\boldsymbol{X}\in\mathbb{R}^{m\times m}$,
such that $\boldsymbol{A}\boldsymbol{A}^{g}=\boldsymbol{X}\begin{bmatrix}\boldsymbol{I}_{r}\\
 & \boldsymbol{0}
\end{bmatrix}\boldsymbol{X}^{-1}$, where $\boldsymbol{I}_{r}\in\mathbb{R}^{r\times r}$ with $r=\text{\emph{rank}}(\boldsymbol{A})$.
Conversely, if there exists a nonsingular $\boldsymbol{X}$ such that
$\boldsymbol{A}\boldsymbol{G}=\boldsymbol{X}\begin{bmatrix}\boldsymbol{I}_{r}\\
 & \boldsymbol{0}
\end{bmatrix}\boldsymbol{X}^{-1}$ for $r=\text{\emph{rank}}(\boldsymbol{A})$, then $\boldsymbol{G}$
is a generalized inverse of $\boldsymbol{A}$.
\end{proposition}

\begin{proof}
$\boldsymbol{A}\boldsymbol{A}^{g}$ is idempotent due to Lemma~\ref{lem:pseudo-inv-range-preservation},
so its eigenvalues are ones and zeros. The invariant subspaces associated
with them are $\mathcal{R}(\boldsymbol{A}\boldsymbol{A}^{g})$ and
$\mathcal{N}(\boldsymbol{A}\boldsymbol{A}^{g})$, respectively. Hence,
$\boldsymbol{A}\boldsymbol{A}^{g}$ is diagonalizable. For the converse
direction, $\boldsymbol{A}\boldsymbol{G}$ is idempotent, and furthermore
$\mathcal{R}(\boldsymbol{A}\boldsymbol{G})=\mathcal{R}(\boldsymbol{A})$
because $\mathcal{R}(\boldsymbol{A}\boldsymbol{G})\subseteq\mathcal{R}(\boldsymbol{A})$
and $\text{rank}(\boldsymbol{A}\boldsymbol{G})=\text{rank}(\boldsymbol{A})=r$.
Hence, $\boldsymbol{G}$ is a generalized inverse due to Lemma~\ref{lem:pseudo-inv-range-preservation}.
\end{proof}
\begin{proposition}
\label{prop:transpose-gen-inv}If $\boldsymbol{A}^{g}$ is a generalized
inverse of $\boldsymbol{A}\in\mathbb{R}^{m\times n}$, then $\left(\boldsymbol{A}^{g}\right)^{T}$
is a generalized inverse of $\boldsymbol{A}^{T}$.
\end{proposition}

\begin{proof}
$\boldsymbol{A}^{T}\left(\boldsymbol{A}^{g}\right)^{T}\boldsymbol{A}^{T}=\left(\boldsymbol{A}\boldsymbol{A}^{g}\boldsymbol{A}\right)^{T}=\boldsymbol{A}^{T}$.
\end{proof}
The above results apply to general $m\times n$ real matrices and
are extensible to complex matrices. In the following, we apply them
only to the special case of real $n\times n$ matrices.

\subsection{\label{subsec:Optimality-condition-consistent}Optimality condition
for consistent systems}

For consistent systems, the generalized inverses and their scalar
multiples constitute the family of optimal RPOs.
\begin{theorem}
\label{thm:optimal-precond} Given $\boldsymbol{A}\in\mathbb{R}^{n\times n}$
and a generalized inverse $\boldsymbol{A}^{g}$, then GMRES with RPO
$\alpha\boldsymbol{A}^{g}$ with $\alpha\neq0$ converges to an LS
solution $\boldsymbol{x}_{\text{LS}}$ of \eqref{eq:least-squares-linear-system}
after one iteration for all $\boldsymbol{b}\in\mathcal{R}(\boldsymbol{A})$
and $\boldsymbol{x}_{0}\in\mathbb{R}^{n}$. Conversely, if GMRES with
RPO $\boldsymbol{G}$ converges to an LS solution in one iteration
for all $\boldsymbol{b}\in\mathcal{R}(\boldsymbol{A})$ and $\boldsymbol{x}_{0}\in\mathbb{R}^{n}$,
then $\boldsymbol{G}$ is a scalar multiple of a generalized inverse
of $\boldsymbol{A}$.
\end{theorem}

\begin{proof}
Since $\boldsymbol{A}\boldsymbol{A}^{g}$ is idempotent, $\mathcal{N}(\boldsymbol{A}\boldsymbol{A}^{g})\cap\mathcal{R}(\boldsymbol{A}\boldsymbol{A}^{g})=\{\boldsymbol{0}\}$.
From Theorem~\ref{thm:consistent-range-asymmetric}, the preconditioned
GMRES converges to an LS solution for all $\boldsymbol{b}\in\mathcal{R}(\boldsymbol{A})$
and $\boldsymbol{x}_{0}\in\mathbb{R}^{n}$. Furthermore, since $\left(\boldsymbol{A}\boldsymbol{A}^{g}\right)^{2}=\boldsymbol{A}\boldsymbol{A}^{g}$,
$\mathcal{K}_{k}(\boldsymbol{A}\boldsymbol{A}^{g},\boldsymbol{b})$
becomes invariant after one iteration. Hence, the preconditioned GMRES
terminates in one iteration. This invariance property also holds if
the RPO is $\alpha\boldsymbol{A}^{g}$ for $\alpha\neq0$.

Conversely, if GMRES with RPO $\boldsymbol{G}$ converges in one step
for all $\boldsymbol{b}\in\mathcal{R}(\boldsymbol{A})$ and $\boldsymbol{x}_{0}\in\mathbb{R}^{n}$,
$\boldsymbol{A}\boldsymbol{G}$ must have only one nonzero eigenvalue
$\lambda$, and $\nicefrac{1}{\lambda}\boldsymbol{A}\boldsymbol{G}$
is idempotent. Furthermore, the solution is always an LS solution
if and only if $\mathcal{R}(\boldsymbol{A}\boldsymbol{G})=\mathcal{R}(\boldsymbol{A})$
\cite[Theorem 3.1]{hayami2010gmres}. Since $\mathcal{R}(\alpha\boldsymbol{A}\boldsymbol{G})=\mathcal{R}(\boldsymbol{A}\boldsymbol{G})$
for $\alpha=\nicefrac{1}{\lambda}\neq0$, from Lemma~\ref{lem:pseudo-inv-range-preservation},
$\boldsymbol{G}$ is a scalar multiple of a generalized inverse of
$\boldsymbol{A}$.
\end{proof}
In Theorem~\ref{thm:optimal-precond}, the LS solution $\boldsymbol{x}_{\text{LS}}$
is equal to $\boldsymbol{A}^{g}\boldsymbol{b}$, which is a generalized-inverse
solution of \eqref{eq:least-squares-linear-system} with $\boldsymbol{b}\in\mathcal{R}(\boldsymbol{A})$.
However, $\boldsymbol{x}_{\text{LS}}$ is not the pseudoinverse solution
unless $\boldsymbol{x}_{\text{LS}}\in\mathcal{R}(\boldsymbol{A}^{T})$,
which is satisfied if $\mathcal{R}(\boldsymbol{A}^{g})=\mathcal{R}(\boldsymbol{A}^{T})$
and $\boldsymbol{x}_{0}\in\mathcal{R}(\boldsymbol{A}^{T})$. 

We assumed exact arithmetic in the above. With rounding errors, $\kappa(\boldsymbol{A}\boldsymbol{A}^{g})$
determines the convergence rate. Even though $\boldsymbol{A}\boldsymbol{A}^{g}$
is idempotent, $\kappa(\boldsymbol{A}\boldsymbol{A}^{g})$ may be
arbitrarily large. For $\boldsymbol{A}^{g}$ to be an ``optimal''
RPO, $\kappa(\boldsymbol{A}\boldsymbol{A}^{g})$ should be small (relative
to $1/\epsilon_{\text{mach}}$) for fast convergence. It can be shown
that $\kappa(\boldsymbol{A}\boldsymbol{A}^{g})\leq\kappa(\boldsymbol{X})$
(see Appendix~\ref{sec:Relationship-of-condition}), so it suffices
to control $\kappa(\boldsymbol{X})$ in practice.

\subsection{\label{subsec:Optimality-conditions-inconsistent}Near-optimality
conditions for inconsistent systems}

If the system is inconsistent, i.e., $\boldsymbol{b}\in\mathbb{R}^{n\times n}\backslash\mathcal{R}(\boldsymbol{A})$,
according to Theorem~\ref{thm:GMRES-range-symmetric}, GMRES with
$\boldsymbol{A}^{g}$ as the right preconditioner may not converge
to an LS solution for some $\boldsymbol{b}\in\mathbb{R}^{n}\backslash\mathcal{R}(\boldsymbol{A})$.
Instead, it converges to a WLS solution, as shown by the following
theorem.
\begin{theorem}
\label{thm:weighted-least-squares}Given $\boldsymbol{A}\in\mathbb{R}^{n\times n}$
and a generalized inverse $\boldsymbol{A}^{g}$, then GMRES with RPO
$\boldsymbol{A}^{g}$ converges to a WLS solution $\boldsymbol{x}_{*}$
of \eqref{eq:least-squares-linear-system} after one iteration for
all $\boldsymbol{b}\in\mathbb{R}^{n}$ and $\boldsymbol{x}_{0}\in\mathbb{R}^{n}$.
Conversely, if the KSP with RPO $\boldsymbol{G}$ is invariant in
one step for all $\boldsymbol{b}\in\mathbb{R}^{n}$ and $\boldsymbol{x}_{0}\in\mathbb{R}^{n}$,
then $\boldsymbol{G}$ is a scalar multiple of a generalized inverse
of $\boldsymbol{A}$.
\end{theorem}

\begin{proof}
We focus on the forward direction since the proof of the converse
direction follows that of Theorem~\ref{thm:optimal-precond}. First,
let us assume that $\boldsymbol{x}_{0}=\boldsymbol{0}$. Since $\boldsymbol{A}\boldsymbol{A}^{g}\boldsymbol{A}=\boldsymbol{A}$,
$\mathcal{K}(\boldsymbol{A}\boldsymbol{A}^{g},\boldsymbol{b})$ becomes
invariant after one iteration. Hence, the preconditioned GMRES terminates
in one iteration, and the solution $y_{*}$ is the LS solution of
\begin{equation}
{\boldsymbol{Q}_{1}^{T}\boldsymbol{A}\boldsymbol{q}_{0}y_{*}\approx\boldsymbol{Q}_{1}^{T}\boldsymbol{b}=\begin{bmatrix}\left\Vert \boldsymbol{b}\right\Vert _{2}\\
\boldsymbol{q}_{1}^{T}\boldsymbol{b}
\end{bmatrix},}\label{eq:ls-solution-gmres}
\end{equation}
and $\boldsymbol{Q}_{1}$ is composed of columns
\begin{align}
\boldsymbol{q}_{0} & =\boldsymbol{b}/\Vert\boldsymbol{b}\Vert_{2},\label{eq:q0}\\
\boldsymbol{q}_{1} & =\boldsymbol{v}/\left\Vert \boldsymbol{v}\right\Vert _{2}\quad\text{with}\quad\boldsymbol{v}=\left(\boldsymbol{I}-\boldsymbol{q}_{0}\boldsymbol{q}_{0}^{T}\right)\boldsymbol{A}\boldsymbol{A}^{g}\boldsymbol{b}.\label{eq:q1}
\end{align}
Hence, $\boldsymbol{x}_{*}=\boldsymbol{q}_{0}y_{*}$ is the LS solution
of 
\begin{equation}
\boldsymbol{Q}_{1}^{T}\boldsymbol{A}\boldsymbol{x}_{*}\approx\boldsymbol{Q}_{1}^{T}\boldsymbol{b}\label{eq:ls-solution-x}
\end{equation}
or, equivalently, a WLS solution that minimizes a seminorm, namely,
\begin{equation}
\left\Vert \boldsymbol{A}\boldsymbol{x}_{*}-\boldsymbol{b}\right\Vert _{\boldsymbol{Q}_{1}\boldsymbol{Q}_{1}^{T}}=\sqrt{\left(\boldsymbol{A}\boldsymbol{x}_{*}-\boldsymbol{b}\right)^{T}\boldsymbol{Q}_{1}\boldsymbol{Q}_{1}^{T}\left(\boldsymbol{A}\boldsymbol{x}_{*}-\boldsymbol{b}\right)}.\label{eq:wls-solution}
\end{equation}
\end{proof}

In Theorem~\ref{thm:weighted-least-squares}, if $\boldsymbol{A}\boldsymbol{A}^{g}$
is not range symmetric, then GMRES would suffer from a hard breakdown
after one iteration from the viewpoint of a true LS solution. The
severity of this breakdown depends on how close $\boldsymbol{A}\boldsymbol{A}^{g}$
is to an orthogonal projector. 

\subsection{\label{subsec:epsilon-accurate-AGI}$\epsilon$-accurate approximate
generalized inverses}

It is impractical to compute an exact generalized inverse, so an approximation
is needed. The following definition establishes a criterion for an
accurate and stable approximation to a generalized inverse.
\begin{definition}
\label{def:epsilon-accurate}Given $\boldsymbol{A}\in\mathbb{R}^{n\times n}$,
$\boldsymbol{G}$ is an $\epsilon$\emph{-accurate} \emph{approximate
generalized inverse} (\emph{AGI}) if there exists $\boldsymbol{X}\in\mathbb{R}^{n\times n}$
as in Proposition~\ref{prop:diagnalizability} such that
\begin{equation}
\left\Vert \boldsymbol{X}^{-1}\boldsymbol{A}\boldsymbol{G}\boldsymbol{X}-\begin{bmatrix}\boldsymbol{I}_{r}\\
 & \boldsymbol{0}
\end{bmatrix}\right\Vert _{2}=\epsilon<1,\label{eq:norm-bound}
\end{equation}
where $\boldsymbol{I}_{r}\in\mathbb{R}^{r\times r}$ with $r=\text{\emph{rank}}(\boldsymbol{A})$.
A class of AGI is $\epsilon$\emph{-accurate} if $\epsilon$ tends
to 0 as its control parameters are tightened. $\boldsymbol{G}$ is
a \emph{stable }AGI if \emph{$\kappa(\boldsymbol{X})\leq C$ for some
$C\ll1/\epsilon_{\text{mach}}$}.
\end{definition}

In the definition, the ``control parameters'' refer to the drop
tolerance and other parameters in an incomplete-factorization algorithm;
see section~\ref{subsec:Hybridizing-multilevel-ILU}. In the following,
$\boldsymbol{\Vert\cdot\Vert}$ without a subscript will refer to
the 2-norm, unless otherwise noted. The bound $C$ depends on the
machine epsilon because the effect of condition numbers as an amplification
factor of rounding errors depends on the specific floating-point arithmetic.
Given an $\epsilon$-accurate preconditioner, if $\kappa(\boldsymbol{X})\approx1$,
then 
\begin{equation}
\kappa(\boldsymbol{A}\boldsymbol{G})\lesssim(1+\epsilon)/(1-\epsilon)\label{eq:bound-condition-number}
\end{equation}
(see Appendix~\ref{sec:Proof-of-bound} for its proof). Hence, assuming
$\kappa(\boldsymbol{X})\approx1$, the smaller $\epsilon$ is, the
better the preconditioner is. Note that for left preconditioning,
one would need to change $\boldsymbol{A}\boldsymbol{G}$ to $\boldsymbol{G}\boldsymbol{A}$
in \eqref{eq:norm-bound}.

The following theorem shows that $\epsilon$\emph{-accurate} preconditioners
guarantee the convergence of preconditioned GMRES for consistent systems
in exact arithmetic.
\begin{theorem}
\label{thm:convergence} GMRES with an $\epsilon$-accurate AGI $\boldsymbol{G}$
of $\boldsymbol{A}$ converges to an LS solution of \eqref{eq:least-squares-linear-system}
in exact arithmetic for all $\boldsymbol{b}\in\mathcal{R}(\boldsymbol{A})$
and any initial guess $\boldsymbol{x}_{0}\in\mathbb{R}^{n}$.
\end{theorem}

To prove the theorem, we first assert the following two lemmas.
\begin{lemma}
\label{lem:AGI-range-preservation}An $\epsilon$-accurate AGI $\boldsymbol{G}$
preserves the range of $\boldsymbol{A}$, i.e., $\mathcal{R}(\boldsymbol{A})=\mathcal{R}(\boldsymbol{A}\boldsymbol{G})$.
\end{lemma}

\begin{proof}
Given $\boldsymbol{X}\in\mathbb{R}^{n\times n}$ in Definition~\ref{def:epsilon-accurate},
due to Proposition~\ref{prop:diagnalizability}, there exists $\boldsymbol{A}^{g}$
such that $\boldsymbol{X}^{-1}\boldsymbol{A}\boldsymbol{A}^{g}\boldsymbol{X}=\begin{bmatrix}\boldsymbol{I}_{r}\\
 & \boldsymbol{0}
\end{bmatrix}$. For all $\boldsymbol{w}\in\mathcal{R}(\boldsymbol{X}^{-1}\boldsymbol{A}\boldsymbol{A}^{g}\boldsymbol{X})=\text{span}\{\boldsymbol{e}_{1},\text{\ensuremath{\boldsymbol{e}_{2}}},\dots\boldsymbol{e}_{r}\}$,
\begin{align*}
\left\Vert \boldsymbol{w}\right\Vert  & =\left\Vert \boldsymbol{X}^{-1}\boldsymbol{A}\boldsymbol{A}^{g}\boldsymbol{X}\boldsymbol{w}\right\Vert \\
 & =\left\Vert \boldsymbol{X}^{-1}(\boldsymbol{A}\boldsymbol{G}-(\boldsymbol{A}\boldsymbol{G}-\boldsymbol{A}\boldsymbol{A}^{g}))\boldsymbol{X}\boldsymbol{w}\right\Vert \\
 & \leq\left\Vert \boldsymbol{X}^{-1}\boldsymbol{A}\boldsymbol{G}\boldsymbol{X}\boldsymbol{w}\right\Vert +\left\Vert \boldsymbol{X}^{-1}\left(\boldsymbol{A}\boldsymbol{G}-\boldsymbol{A}\boldsymbol{A}^{g}\right)\boldsymbol{X}\boldsymbol{w}\right\Vert \\
 & \leq\left\Vert \boldsymbol{X}^{-1}\boldsymbol{A}\boldsymbol{G}\boldsymbol{X}\boldsymbol{w}\right\Vert +\epsilon\left\Vert \boldsymbol{w}\right\Vert ,
\end{align*}
where $\epsilon<1$ due to Definition~\ref{def:epsilon-accurate}.
Hence, $\boldsymbol{X}^{-1}\boldsymbol{A}\boldsymbol{G}\boldsymbol{X}\boldsymbol{w}\neq\boldsymbol{0}$
if $\boldsymbol{w}\neq\boldsymbol{0}$. In other words, 
\[
\mathcal{N}(\boldsymbol{X}^{-1}\boldsymbol{A}\boldsymbol{G}\boldsymbol{X})\cap\mathcal{R}(\boldsymbol{X}^{-1}\boldsymbol{A}\boldsymbol{A}^{g}\boldsymbol{X})=\{\boldsymbol{0}\}.
\]
Therefore, $\text{dim}(\mathcal{N}(\boldsymbol{X}^{-1}\boldsymbol{A}\boldsymbol{G}\boldsymbol{X}))\leq n-r$
and $\text{rank}(\boldsymbol{X}^{-1}\boldsymbol{A}\boldsymbol{G}\boldsymbol{X})\geq r=\text{\text{rank}(\ensuremath{\boldsymbol{A}})}$.
Since $\text{rank}(\boldsymbol{X}^{-1}\boldsymbol{A}\boldsymbol{G}\boldsymbol{X})\leq\text{rank}(\boldsymbol{A})$,
$\text{rank}(\boldsymbol{A}\boldsymbol{G})=\text{rank}(\boldsymbol{X}^{-1}\boldsymbol{A}\boldsymbol{G}\boldsymbol{X})=\text{rank}(\boldsymbol{A})$,
which along with $\mathcal{R}(\boldsymbol{A}\boldsymbol{G})\subseteq\mathcal{R}(\boldsymbol{A})$
implies that $\mathcal{R}(\boldsymbol{A})=\mathcal{R}(\boldsymbol{A}\boldsymbol{G})$.
\end{proof}
\begin{lemma}
\label{lem:AGI-index-1}Given an $\epsilon$-accurate AGI $\boldsymbol{G}$,
$\mathcal{R}(\boldsymbol{A}\boldsymbol{G})\cap\mathcal{N}(\boldsymbol{A}\boldsymbol{G})=\{\boldsymbol{0}\}$.
\end{lemma}

\begin{proof}
Consider $\boldsymbol{w}\in\mathcal{R}(\boldsymbol{A}\boldsymbol{G})\backslash\{\boldsymbol{0}\}$,
where $\mathcal{R}(\boldsymbol{A}\boldsymbol{G})=\mathcal{R}(\boldsymbol{A})=\mathcal{R}(\boldsymbol{A}\boldsymbol{A}^{g})$
due to Lemmas~\ref{lem:pseudo-inv-range-preservation} and \ref{lem:AGI-range-preservation}.
Since $\boldsymbol{A}\boldsymbol{A}^{g}$ is idempotent, $\boldsymbol{A}\boldsymbol{A}^{g}\boldsymbol{w}=\boldsymbol{w}$.
Let $\boldsymbol{v}=\boldsymbol{X}^{-1}\boldsymbol{w}$, where $\boldsymbol{X}$
is defined in Definition~\ref{def:epsilon-accurate}, and then $\boldsymbol{X}^{-1}\boldsymbol{A}\boldsymbol{A}^{g}\boldsymbol{X}\boldsymbol{v}=\boldsymbol{X}^{-1}\boldsymbol{A}\boldsymbol{A}^{g}\boldsymbol{w}=\boldsymbol{X}^{-1}\boldsymbol{w}=\boldsymbol{v}.$
Hence, 
\begin{align*}
\boldsymbol{X}^{-1}\boldsymbol{A}\boldsymbol{G}\boldsymbol{w} & =\boldsymbol{X}^{-1}\boldsymbol{A}\boldsymbol{G}\boldsymbol{X}\boldsymbol{v}\\
 & =\boldsymbol{X}^{-1}\boldsymbol{A}\boldsymbol{A}^{g}\boldsymbol{X}\boldsymbol{v}+\boldsymbol{X}^{-1}(\boldsymbol{A}\boldsymbol{G}-\boldsymbol{A}\boldsymbol{A}^{g})\boldsymbol{X}\boldsymbol{v}\\
 & =\boldsymbol{v}+\boldsymbol{X}^{-1}(\boldsymbol{A}\boldsymbol{G}-\boldsymbol{A}\boldsymbol{A}^{g})\boldsymbol{X}\boldsymbol{v},
\end{align*}
where $\left\Vert \boldsymbol{X}^{-1}(\boldsymbol{A}\boldsymbol{G}-\boldsymbol{A}\boldsymbol{A}^{g})\boldsymbol{X}\boldsymbol{v}\right\Vert <\left\Vert \boldsymbol{v}\right\Vert $
due to Definition~\ref{def:epsilon-accurate}. Therefore, $\boldsymbol{X}^{-1}\boldsymbol{A}\boldsymbol{G}\boldsymbol{w}\neq\boldsymbol{0}$
and $\boldsymbol{w}\not\in\mathcal{N}(\boldsymbol{A}\boldsymbol{G})$.
\end{proof}
Given these two lemmas, it is now straightforward to prove Theorem~\ref{thm:convergence}.
\begin{proof}
Due to Lemmas~\ref{lem:AGI-range-preservation} and \ref{lem:AGI-index-1},
given an $\epsilon$-accurate $\boldsymbol{G}$, $\boldsymbol{A}\boldsymbol{G}$
satisfy the conditions in Theorem~\ref{thm:consistent-range-asymmetric}.
Hence, the preconditioned GMRES converges to an LS solution of \eqref{eq:least-squares-linear-system}.
\end{proof}
Note that Theorem~\ref{thm:convergence} is not a necessary condition
for convergence. For example, $\boldsymbol{A}^{T}$ is not an AGI
but AB-GMRES converges if one uses $\boldsymbol{A}^{T}$ as $\boldsymbol{B}$
\cite{hayami2010gmres}. However, Theorem~\ref{thm:convergence}
provides us the guideline for designing nearly optimal preconditioners
for consistent systems, which we will address next.

\section{\label{sec:HIF}Hybrid incomplete factorization for consistent systems}

In this section, we introduce an efficient and robust preconditioner
for consistent asymmetric systems. We start by reviewing HILUCSI for
nonsingular systems \cite{chen2021hilucsi} and showing that it is
unstable for singular systems. We then introduce HIF by overcoming
the instabilities in HILUCSI.

\subsection{\label{subsec:MLILU}Preliminary: HILUCSI for nonsingular systems}

\emph{HILUCSI} stands for hierarchical incomplete LU-Crout with scalability
and inverse-based droppings. Similar to ILUPACK \cite{bollhofer2011ilupack},
HILUCSI computes an MLILU factorization, which is a general algebraic
framework for building block preconditioners. More precisely, let
$\boldsymbol{A}$ be the input coefficient matrix. A two-level ILU
reads
\begin{equation}
\boldsymbol{P}^{T}\boldsymbol{W}\boldsymbol{A}\boldsymbol{V}\boldsymbol{Q}=\begin{bmatrix}\boldsymbol{B} & \boldsymbol{F}\\
\boldsymbol{E} & \boldsymbol{C}
\end{bmatrix}\approx\tilde{\boldsymbol{M}}=\begin{bmatrix}\tilde{\boldsymbol{B}} & \tilde{\boldsymbol{F}}\\
\tilde{\boldsymbol{E}} & \boldsymbol{C}
\end{bmatrix}=\begin{bmatrix}\boldsymbol{L}_{B}\\
\boldsymbol{L}_{E} & \boldsymbol{I}
\end{bmatrix}\begin{bmatrix}\boldsymbol{D}_{B}\\
 & \boldsymbol{S}
\end{bmatrix}\begin{bmatrix}\boldsymbol{U}_{B} & \boldsymbol{U}_{F}\\
 & \boldsymbol{I}
\end{bmatrix},\label{eq:two-level-ILU}
\end{equation}
where $\boldsymbol{B}\approx\tilde{\boldsymbol{B}}=\boldsymbol{L}_{B}\boldsymbol{D}_{B}\boldsymbol{U}_{B}$
corresponds to an incomplete factorization of the leading block, and
$\boldsymbol{S}=\boldsymbol{C}-\boldsymbol{L}_{E}\boldsymbol{D}_{B}\boldsymbol{U}_{F}$
is the Schur complement. $\boldsymbol{P}$ and $\boldsymbol{Q}$ correspond
to row and column permutation matrices, respectively, and $\boldsymbol{W}$
and $\boldsymbol{V}$ are row and column scaling diagonal matrices,
respectively. These permutation and scaling matrices are side-products
of some preprocessing steps (such as using equilibration \cite{duff2001algorithms}
and reordering \cite{amestoy1996approximate}) and dynamic pivoting
strategies. For the two-level ILU in \eqref{eq:two-level-ILU}, $\boldsymbol{W}^{-1}\boldsymbol{P}\tilde{\boldsymbol{M}}\boldsymbol{Q}^{T}\boldsymbol{V}$
provides a preconditioner of $\boldsymbol{A}$. Given a block vector
$\boldsymbol{u}=\begin{bmatrix}\boldsymbol{u}_{1}\\
\boldsymbol{u}_{2}
\end{bmatrix}$, where $\boldsymbol{u}_{1}$ and $\boldsymbol{u}_{2}$ corresponds
to $\tilde{\boldsymbol{B}}$ and $\boldsymbol{C}$, respectively,
\begin{equation}
\tilde{\boldsymbol{M}}^{-1}\vec{u}=\begin{bmatrix}\tilde{\boldsymbol{B}}^{-1}\boldsymbol{u}_{1}\\
\vec{0}
\end{bmatrix}+\begin{bmatrix}-\tilde{\boldsymbol{B}}^{-1}\tilde{\boldsymbol{F}}\\
\vec{I}
\end{bmatrix}\boldsymbol{S}^{-1}(\boldsymbol{u}_{2}-\tilde{\boldsymbol{E}}\tilde{\boldsymbol{B}}^{-1}\boldsymbol{u}_{1}).\label{eq:Schur-complement-inverse}
\end{equation}
By factorizing the Schur complement $\boldsymbol{S}$ recursively,
we then obtain an $m$-level ILU and a corresponding multilevel preconditioner,
namely,
\begin{equation}
\boldsymbol{M}=\underset{\boldsymbol{L}}{\underbrace{\boldsymbol{W}_{1}^{-1}\boldsymbol{P}_{1}\boldsymbol{L}_{1}\cdots\boldsymbol{W}_{m}^{-1}\boldsymbol{P}_{m}\boldsymbol{L}_{m}}}\begin{bmatrix}\boldsymbol{D}\\
 & \boldsymbol{S}_{m}
\end{bmatrix}\underset{\boldsymbol{U}}{\underbrace{\boldsymbol{U}_{m}\boldsymbol{Q}_{m}^{T}\boldsymbol{V}_{m}^{-1}\cdots\boldsymbol{U}_{1}\boldsymbol{Q}_{1}^{T}\boldsymbol{V}_{1}^{-1}}},\label{eq:multilevel-ILU}
\end{equation}
where $\boldsymbol{L}_{k}\in\mathbb{R}^{n\times n}$ denotes $\begin{bmatrix}\ensuremath{\boldsymbol{I}}_{n-n_{k}}\\
 & \boldsymbol{L}_{B}\\
 & \boldsymbol{L}_{E} & \boldsymbol{I}
\end{bmatrix}$ for $\begin{bmatrix}\boldsymbol{L}_{B}\\
\boldsymbol{L}_{E} & \boldsymbol{I}
\end{bmatrix}\in\mathbb{R}^{n_{k}\times n_{k}}$ in \eqref{eq:two-level-ILU} at the $k$th level (similarly for $\boldsymbol{P}_{k}$,
$\boldsymbol{W}_{k}$, $\boldsymbol{Q}_{k}$, $\boldsymbol{V}_{k}$,
and $\boldsymbol{U}_{k}$), $\boldsymbol{D}\in\mathbb{R}^{d\times d}$
is composed of the ``union'' of $\boldsymbol{D}_{B}$ from all $m$
levels, and $\boldsymbol{S}_{m}$ is the final Schur complement. MLILU
factorizes $\boldsymbol{S}_{m}$ using a complete LU factorization
with pivoting, assuming its size is small enough.

The multilevel framework above is shared by HILUCSI \cite{chen2021hilucsi},
ILUPACK \cite{bollhofer2011ilupack}, ILU++ \cite{mayer2007multilevel},
etc. However, the multilevel structures are obtained using different
strategies for different methods. HILUCSI constructs the levels using
two \emph{deferring} strategies at each level. First, after applying
equilibration (in particular, using MC64 \cite{duff2001algorithms}),
a row and its corresponding column are \emph{statically} deferred
to the next level if its diagonal entry is zero (or close to machine
precision). Second, during factorization at the $k$th level, if a
diagonal entry causes ill-conditioned triangular (i.e., $\boldsymbol{L}_{k}$
and $\boldsymbol{U}_{k}$) or diagonal (i.e., $\boldsymbol{D}_{B}$)
factors, then we defer the row and its corresponding column dynamically
to the next level. For efficiency, HILUCSI introduced a scalability-oriented
dropping strategy to achieve near-linear time complexity. To achieve
robustness, it also employs some well-known techniques, including
the Crout version of ILU factorization \cite{li2003crout}, inverse-based
dropping \cite{bollhofer2006multilevel}, fill-reduction reordering
\cite{amestoy1996approximate}, etc. For additional technical details
on HILUCSI, we refer readers to \cite{chen2021hilucsi}.

\begin{remark}The notion of $\epsilon$-accuracy applies to HILUCSI
for nonsingular systems by considering $\boldsymbol{A}^{g}=\boldsymbol{A}^{-1}$.
To see this, given an $\boldsymbol{M}$ in \eqref{eq:multilevel-ILU},
let $\boldsymbol{M}_{*}$ and $\boldsymbol{L}_{*}$ be the corresponding
matrices in \eqref{eq:multilevel-ILU} without dropping but with the
same permutation and scaling factors as for $\boldsymbol{M}$. Then
$\boldsymbol{A}\boldsymbol{A}^{-1}=\boldsymbol{A}\boldsymbol{M}_{*}^{-1}=\boldsymbol{L}_{*}\boldsymbol{I}\boldsymbol{L}_{*}^{-1}$,
and $\boldsymbol{M}$ is $\epsilon$-accurate with respect to $\boldsymbol{A}^{g}=\boldsymbol{A}^{-1}$
and $\boldsymbol{X}=\boldsymbol{L}_{*}$. HILUCSI ensures the well-conditioning
of $\boldsymbol{L}$ through static and dynamic deferring, so $\boldsymbol{L}_{*}$
is expected to be well-conditioned if the perturbations from the droppings
are sufficiently small. This analysis based on $\epsilon$-accuracy
explains the robustness of HILUCSI-preconditioned GMRES for nonsingular
systems in \cite{chen2021hilucsi}. For singular systems, however,
HILUCSI sometimes fails, which we address next.\end{remark}

\subsection{\label{subsec:Insufficiency-of-HILUCSI}Insufficiency of HILUCSI
for singular systems}

The following proposition will help identify the potential causes
of failures of HILUCSI for singular systems \cite{chen2021hilucsi}.
\begin{proposition}
\label{prop:breakdown}In the $m$-level ILU in \eqref{eq:multilevel-ILU},
if $\boldsymbol{S}_{m}$ is factorized using LU factorization with
pivoting, the classical 2-norm condition number of $\boldsymbol{M}$
is bounded by 
\begin{equation}
\kappa(\boldsymbol{M})\leq\kappa\left(\begin{bmatrix}\boldsymbol{D}\\
 & \boldsymbol{S}_{m}
\end{bmatrix}\right)\prod_{k}\left(\kappa(\boldsymbol{L}_{k})\kappa(\boldsymbol{U}_{k})\right)\prod_{k}\left(\kappa(\boldsymbol{W}_{k})\kappa(\boldsymbol{V}_{k})\right).\label{eq:cond-bound-mlevel}
\end{equation}
\end{proposition}

\begin{proof}
The inequality follows from the submultiplicative property of the
2-norm. 
\end{proof}
In \eqref{eq:cond-bound-mlevel}, $\kappa\left(\begin{bmatrix}\boldsymbol{D}\\
 & \boldsymbol{S}_{m}
\end{bmatrix}\right)=\max\{\Vert\boldsymbol{D}\Vert,\Vert\boldsymbol{S}_{m}\Vert\}\max\{\Vert\boldsymbol{D}^{-1}\Vert,\Vert\boldsymbol{S}_{m}^{-1}\Vert\}$. If the condition numbers of all the triangular factors are $\mathcal{O}(1)$
and the magnitudes of all the entries in $\boldsymbol{D}$, $\boldsymbol{W}_{k}$,
and $\boldsymbol{V}_{k}$ are close to $1$, then a large $\kappa(\boldsymbol{M})$
(e.g., $\kappa(\boldsymbol{M})=\mathcal{O}(1/\epsilon_{\text{mach}})$)
implies a large $\kappa(\boldsymbol{S}_{m})$. In this case, the LU
factorization of $\boldsymbol{S}_{m}$ may blow up. In general, if
$\boldsymbol{A}$ is singular and dropping is minimal, then $\boldsymbol{M}\approx\boldsymbol{A}$,
and $\boldsymbol{M}$ may be extremely ill-conditioned due to an ill-conditioned
$\boldsymbol{S}_{m}$, which is likely to happen if $\text{dim}(\mathcal{N}(\boldsymbol{A}))\gg1$.

From \eqref{eq:cond-bound-mlevel}, it is also evident that the preconditioner
may be unstable if the scaling factors $\boldsymbol{W}_{k}$ and $\boldsymbol{V}_{k}$
are too large or too small. In HILUCSI, we obtain the scaling factors
using MC64 \cite{duff2001algorithms}, which computes a maximum weighted
matching on the bipartite graph $(R,C,E)$, where $R$, $C$, and
$E$ correspond to the rows, columns, and nonzeros of $\boldsymbol{A}$,
respectively. Such a matching does not exist for structurally singular
systems \cite{duff2005strategies}, so the scaling factors are ill-defined.
As a result, the scaling factors may be susceptible to perturbation
for ``nearly'' structurally singular systems. In our experiments,
we observed that the scaling factors from MC64 can be abnormally large
($\gg1/\epsilon_{\text{mach}}$) or small ($\ll1/\epsilon_{\text{mach}}$)
for some singular systems, especially on the Schur complements for
singular systems.

\subsection{\label{subsec:Hybridizing-multilevel-ILU}Hybridizing multilevel
ILU and RRQR}

To develop a stable $\epsilon$-accurate preconditioner, we hybridize
HILUCSI with RRQR. In particular, we apply QRCP on the Schur complement
$\boldsymbol{S}_{m}$ in \eqref{eq:multilevel-ILU} to obtain 
\begin{equation}
\boldsymbol{S}_{m}\boldsymbol{P}=\boldsymbol{Q}\begin{bmatrix}\boldsymbol{R}_{1} & \boldsymbol{R}_{2}\\
 & \boldsymbol{0}
\end{bmatrix},\label{eq:QRCP-Schur-Complement}
\end{equation}
where $\boldsymbol{R}_{1}\in\mathbb{R}^{n_{s}\times n_{s}}$ with
$n_{s}=\text{rank}(\boldsymbol{S}_{m})$, $r_{11}\geq r_{22}\geq\dots\ge r_{n_{s}n_{s}}>0$
along its diagonal, and $\kappa(\boldsymbol{R}_{1})\ll1/\epsilon_{\text{mach}}$.
Let the RPO be 
\begin{equation}
\boldsymbol{G}=\boldsymbol{M}^{g}=\boldsymbol{U}^{-1}\begin{bmatrix}\boldsymbol{D}^{-1}\\
 & \boldsymbol{S}_{m}^{g}
\end{bmatrix}\boldsymbol{L}^{-1}=\boldsymbol{U}^{-1}\begin{bmatrix}\boldsymbol{D}^{-1}\\
 & \boldsymbol{P}\begin{bmatrix}\boldsymbol{R}_{1}^{-1} & \boldsymbol{0}\\
\boldsymbol{} & \boldsymbol{0}
\end{bmatrix}\boldsymbol{Q}^{T}
\end{bmatrix}\boldsymbol{L}^{-1},\label{eq:mlilu-pseudoinverse}
\end{equation}
where $\boldsymbol{L}$ and $\boldsymbol{U}$ are the same as those
in \eqref{eq:multilevel-ILU}. We refer to this new preconditioner
as \emph{hybrid incomplete factorization}, or \emph{HIF}. The following
lemma shows that if there is no dropping in LU, HIF enables optimal
convergence of GMRES in exact arithmetic.
\begin{lemma}
\label{lem:range-preservation} If no dropping is applied in the ILU
portion of HIF, then $\boldsymbol{G}$ in \eqref{eq:mlilu-pseudoinverse}
is a generalized inverse of $\boldsymbol{A}$.
\end{lemma}

\begin{proof}
Suppose $\boldsymbol{D}\in\mathbb{R}^{d\times d}$. By construction,
\begin{equation}
\boldsymbol{A}\boldsymbol{G}=\boldsymbol{L}\begin{bmatrix}\boldsymbol{D}\\
 & \boldsymbol{S}_{m}
\end{bmatrix}\boldsymbol{U}\boldsymbol{U}^{-1}\begin{bmatrix}\boldsymbol{D}^{-1}\\
 & \boldsymbol{S}_{m}^{g}
\end{bmatrix}\boldsymbol{L}^{-1}=\boldsymbol{L}\begin{bmatrix}\boldsymbol{I}_{d}\\
 & \boldsymbol{S}_{m}\boldsymbol{S}_{m}^{g}
\end{bmatrix}\boldsymbol{L}^{-1}.\label{eq:similarity-transformation}
\end{equation}
Hence, the eigenvalues of $\boldsymbol{A}\boldsymbol{G}$ have the
same multiplicities as those in $\begin{bmatrix}\boldsymbol{I}_{d}\\
 & \boldsymbol{S}_{m}\boldsymbol{S}_{m}^{g}
\end{bmatrix}$, where $\boldsymbol{S}_{m}=\boldsymbol{Q}\begin{bmatrix}\boldsymbol{R}_{1} & \boldsymbol{R}_{2}\\
 & \boldsymbol{0}
\end{bmatrix}\boldsymbol{P}^{T}$. Note that $\boldsymbol{S}_{m}^{g}=\hat{\boldsymbol{P}}\boldsymbol{R}_{1}^{-1}\hat{\boldsymbol{Q}}^{T}=\boldsymbol{P}\begin{bmatrix}\boldsymbol{R}_{1}^{-1} & \boldsymbol{0}\\
\boldsymbol{} & \boldsymbol{0}
\end{bmatrix}\boldsymbol{Q}^{T}$, where $\hat{\boldsymbol{P}}$ and $\hat{\boldsymbol{Q}}$ are composed
of the first $r$ columns of $\boldsymbol{P}$ and $\boldsymbol{Q}$.
Then, 
\begin{equation}
\boldsymbol{S}_{m}\boldsymbol{S}_{m}^{g}=\boldsymbol{Q}\begin{bmatrix}\boldsymbol{R}_{1} & \boldsymbol{R}_{2}\\
\boldsymbol{} & \boldsymbol{0}
\end{bmatrix}\begin{bmatrix}\boldsymbol{R}_{1}^{-1} & \boldsymbol{0}\\
 & \boldsymbol{0}
\end{bmatrix}\boldsymbol{Q}^{T}=\boldsymbol{Q}\begin{bmatrix}\boldsymbol{I}_{d}\\
 & \boldsymbol{0}
\end{bmatrix}\boldsymbol{Q}^{T}=\hat{\boldsymbol{Q}}\hat{\boldsymbol{Q}}^{T}\label{eq:pseudoinverse-eigenvalues}
\end{equation}
and 
\begin{equation}
\boldsymbol{A}\boldsymbol{G}\boldsymbol{A}=\boldsymbol{L}\begin{bmatrix}\boldsymbol{I}_{d}\\
 & \hat{\boldsymbol{Q}}\hat{\boldsymbol{Q}}^{T}
\end{bmatrix}\boldsymbol{L}^{-1}\boldsymbol{L}\begin{bmatrix}\boldsymbol{D}\\
 & \boldsymbol{S}_{m}
\end{bmatrix}\boldsymbol{U}=\boldsymbol{L}\begin{bmatrix}\boldsymbol{D}\\
 & \hat{\boldsymbol{Q}}\hat{\boldsymbol{Q}}^{T}\boldsymbol{S}_{m}
\end{bmatrix}\boldsymbol{U}=\boldsymbol{A}.\label{eq:AGA}
\end{equation}
Therefore, $\boldsymbol{G}$ is a generalized inverse of $\boldsymbol{A}$.
\end{proof}
\begin{remark} Lemma~\ref{lem:range-preservation} also holds if
we replace QRCP with a different rank-revealing decomposition (such
as TSVD \cite[p. 291]{Golub13MC}) on the Schur complement $\boldsymbol{S}_{m}$.
We use QRCP for its efficiency compared to TSVD, in case the final
Schur complement is relatively large. \end{remark}

In practice, dropping is required for efficiency. For $\boldsymbol{G}$
to be $\epsilon$-accurate with dropping, all the factors must be
stable. To this end, we impose additional safeguards on the scaling
factors from MC64 for singular systems. In particular, we consider
the scaling factors of a row and its corresponding column unstable
if their ratios with the maximum entries in the row and column deviate
too much from unity (e.g., their ratio exceeds $1000$). After identifying
the unstable factors, we defer their corresponding rows and columns
to the next level along with the zero diagonals during static pivoting.
In terms of the Schur complement $\boldsymbol{S}_{m}$, we switch
to using QRCP if $\boldsymbol{S}_{m}$ is small enough or static deferring
defers virtually all rows and columns.  We assert the following convergence
property of GMRES with HIF.
\begin{corollary}
\label{cor:HIF-convergence} With sufficiently small droppings in
HIF, $\boldsymbol{G}$ in \eqref{eq:mlilu-pseudoinverse} is an $\epsilon$-accurate
AGI, and GMRES with RPO $\boldsymbol{G}$ does not break down before
finding an LS solution of \eqref{eq:least-squares-linear-system}
for all $\boldsymbol{b}\in\mathcal{R}(\boldsymbol{A})$ and $\boldsymbol{x}_{0}\in\mathbb{R}^{n}$
in exact arithmetic.
\end{corollary}

\begin{proof}
Let $\boldsymbol{A}^{g}$ be the generalized inverse of $\boldsymbol{A}$
with the same permutation and scaling factors as $\boldsymbol{G}$.
Let $\boldsymbol{A}\boldsymbol{A}^{g}=\boldsymbol{X}\begin{bmatrix}\boldsymbol{I}_{r}\\
 & \boldsymbol{0}
\end{bmatrix}\boldsymbol{X}^{-1}$, where $r=\text{rank}(\boldsymbol{A})$ and $\boldsymbol{X}=\boldsymbol{L}\begin{bmatrix}\boldsymbol{I}_{d}\\
 & \boldsymbol{Q}_{m}
\end{bmatrix}$ as in \eqref{eq:AGA}. Then, 
\begin{align}
\left\Vert \boldsymbol{X}^{-1}\boldsymbol{A}\boldsymbol{G}\boldsymbol{X}-\begin{bmatrix}\boldsymbol{I}_{r}\\
 & \boldsymbol{0}
\end{bmatrix}\right\Vert  & =\left\Vert \boldsymbol{X}^{-1}\left(\boldsymbol{A}\boldsymbol{G}-\boldsymbol{A}\boldsymbol{A}^{g}\right)\boldsymbol{X}\right\Vert \leq\kappa(\boldsymbol{X})\Vert\boldsymbol{A}\Vert\Vert\boldsymbol{G}-\boldsymbol{A}^{g}\Vert,\label{eq:bound-epsilon-accuracy}
\end{align}
which is bounded by $1$ for sufficiently small droppings because
all the components in HIF have an $\mathcal{O}(1)$ condition number.
Convergence then follows from Theorem~\ref{thm:convergence}.
\end{proof}
In terms of the implementation, QRCP is available in LAPACK \cite{anderson1999lapack}
as \textsf{xGEQP3} (where x is s and d for single and double precision,
respectively). One needs to apply a post-processing step to truncate
the rightmost columns to resolve ill-conditioning, which can be done
incrementally by using Bischof's algorithm \cite{bischof1990incremental}
to estimate $\kappa(\boldsymbol{R}_{1:k,1:k})$, e.g., using \textsf{xLAIC1
}in LAPACK. We found that it suffices to use $10^{10}$ in double
precision (approximately $\epsilon_{\text{mach}}^{-2/3}$) as the
upper bound on $\kappa(\boldsymbol{R}_{1:k,1:k})$. We will present
numerical results in section~\ref{subsec:Solution-of-consistent}.

\section{\label{sec:HIFIR}HIF with iterative refinement}

Iterative refinement is a well-known technique in direct and iterative
solvers for ill-conditioned systems \cite{arioli2009using,carson2017new,Golub13MC,skeel1980iterative}.
In this section, we use it with HIF to construct variable preconditioners
for the computation of null-space vectors in section~\ref{sec:Application-to-null-space}.

Given an AGI, i.e., $\boldsymbol{G}\approx\boldsymbol{A}^{g}$, starting
from $\boldsymbol{x}_{0}\in\mathbb{R}^{n}$, we refine the solution
iteratively by obtaining $\boldsymbol{x}_{j}$ for $j=1,2,\dots$
as
\begin{align}
\boldsymbol{r}_{j-1} & =\boldsymbol{b}-\boldsymbol{A}\boldsymbol{x}_{j-1},\label{eq:residual}\\
\boldsymbol{x}_{j} & =\boldsymbol{G}\boldsymbol{r}_{j-1}+\boldsymbol{x}_{j-1}.\label{eq:iter-refinement}
\end{align}
This process can be interpreted as a fixed-point iteration, i.e.,
\begin{equation}
\boldsymbol{x}_{j}=\boldsymbol{G}(\boldsymbol{b}-\boldsymbol{A}\boldsymbol{x}_{j-1})+\boldsymbol{x}_{j-1}=(\boldsymbol{I}-\boldsymbol{G}\boldsymbol{A})\boldsymbol{x}_{j-1}+\boldsymbol{G}\boldsymbol{b}.\label{eq:fixed-point-iteration}
\end{equation}
We refer to the process as \emph{HIF with iterative refinement}, or
\emph{HIFIR}.

\begin{remark}Fixed-point iterations are often referred to as \emph{stationary
iterative methods}. Some classical methods for singular systems include
Karcmarz's relaxation for consistent systems \cite{Karczmarz1937}
and Tanabe's extension for inconsistent systems \cite{Tanabe1971}
(see also \cite{Popa1995}). Stationary iterations typically have
a splitting $\boldsymbol{A}=\boldsymbol{M}-\boldsymbol{N}$ (or $\boldsymbol{A}^{T}\boldsymbol{A}=\boldsymbol{M}-\boldsymbol{N}$
if applied to the normal equation \cite{Elfving1998}). When $\boldsymbol{G}$
is nonsingular, HIFIR is a stationary iteration with $\boldsymbol{M}=\boldsymbol{G}^{-1}$.
When $\boldsymbol{G}$ is singular, in general HIFIR is not a stationary
iteration with a simple splitting. \end{remark}

As a standalone solver, HIFIR converges for all $\boldsymbol{x}_{0}\in\mathbb{R}^{n}$
if and only if the spectral radius of its iteration matrix, i.e.,
$\rho(\boldsymbol{I}-\boldsymbol{G}\boldsymbol{A})$, is less than
1. If $\boldsymbol{A}$ is singular, then $\boldsymbol{G}\boldsymbol{A}$
has a zero eigenvalue and $\rho(\boldsymbol{I}-\boldsymbol{G}\boldsymbol{A})\geq1$,
so HIFIR does not converge in general. For example, if $\boldsymbol{x}_{0}$
has a nonzero component in $\mathcal{N}(\boldsymbol{G}\boldsymbol{A})\backslash\mathcal{R}(\boldsymbol{G})\backslash\{\boldsymbol{0}\}$,
then this component may not diminish during the iterative refinement.
Nevertheless, HIFIR converges for some cases. For completeness, we
present an analysis in Proposition~\ref{prop:iterative-refinement}.
For systems that satisfy Proposition~\ref{prop:iterative-refinement},
KSP methods would converge in just one iteration with a sufficiently
large number of iterations in HIFIR. Hence, HIFIR may improve the
accuracy of the preconditioner with a similar effect as tightening
the dropping thresholds in HIF, without increasing memory usage. However,
 too many refinement steps may slow down convergence. In the worst
case, if $\boldsymbol{G}$ is very close to a generalized inverse
$\boldsymbol{A}^{g}$ where $\boldsymbol{A}\boldsymbol{A}^{g}$ is
range asymmetric, then FGMRES may converge to a WLS (instead of an
LS) solution for singular systems as in Theorem~\ref{thm:weighted-least-squares}.
This phenomenon is analogous to the ``over-solve'' of the inner
iterations in inexact Newton's methods \cite{tuminaro2002backtracking},
so we also refer to it as \emph{over-solve}.   

In terms of implementation, at the $k$th step of FGMRES, given the
last vector $\boldsymbol{q}$ from the generalized Arnoldi process,
HIFIR solves $\boldsymbol{A}\boldsymbol{x}_{j}\approx\boldsymbol{q}$
iteratively starting with $\boldsymbol{x}_{0}=\boldsymbol{0}$, and
the iterative-refinement process (i.e., the inner iteration) repeats
until $\left\Vert \boldsymbol{r}_{j}\right\Vert /\left\Vert \boldsymbol{r}_{0}\right\Vert \not\in[\beta_{L},\beta_{U}]$
or $j=\text{maxiter}$. Here, the $\beta$ values guard against over-solve
for converging and diverging cases, respectively. In practice, we
found that $\beta_{L}=0.2$ and $\beta_{U}=100$ worked reasonably
well. For $\text{maxiter}$, we start with 16 and double it iteratively
after each restart of restarted FGMRES.

The convergence conditions in Proposition~\ref{prop:iterative-refinement}
are very restrictive. In addition, HIFIR is not robust as a variable
preconditioner for FGMRES for solving inconsistent systems because
it cannot introduce new components into the FKSP beyond those in the
KSP with $\boldsymbol{G}$ as a fixed preconditioner, except for the
random perturbations from rounding errors. To avoid these nonrobustness
issues, we shift our attention to its use as a variable preconditioner
in the context of computing null-space vectors.

\section{\label{sec:Application-to-null-space}Application to null-space vectors
and pseudoinverse solutions}

In this section, we apply HIFIR-preconditioned FGMRES as the core
component in solving two challenging linear-algebra problems, namely
computing multiple null-space vectors of a singular matrix and computing
the pseudoinverse solution of an inconsistent linear system.

\subsection{\label{subsec:finding-null-spaces} Computation of null-space vectors}

In this application, we compute orthonormal basis vectors of the null
space of a matrix $\boldsymbol{A}$ to near machine precision. This
problem is essential in the context of finding an eigenvector (or
a singular vector) from a given eigenvalue (or a singular value) \cite[Section 5]{choi2006iterative},
which is fundamental in mathematics and physics \cite{denton2021eigenvectors}.
It will also enable us to compute the pseudoinverse solutions for
systems with low-dimensional null spaces in section~\ref{subsec:pipit}.
This problem is challenging due to potential ill-conditioning when
there are near-$\epsilon_{\text{mach}}$ eigenvalues or when the angles
between the left and right null spaces are large, as implications
of the perturbation theory of eigenvectors \cite[Section 7.2]{Golub13MC}. 

To compute a set of orthonormal basis vectors of the null space, we
solve a series of inconsistent systems using HIFIR-preconditioned
FGMRES. Note that this algorithm involves a triple loop. We will use
$j$ and $k$ for the indices of the iterations in HIFIR and FGMRES,
respectively, which correspond to the inner-most and intermediate
loops, and we will use $i$ for the index of the null-space basis
function (i.e., the outermost loop). For the $i$th null-space basis
vector, we solve the LS system
\begin{equation}
\boldsymbol{y}_{i}=\arg\min_{\boldsymbol{y}}\left\Vert \tilde{\boldsymbol{A}}\boldsymbol{y}-\boldsymbol{b}_{i}\right\Vert ,\qquad\text{where}\qquad\tilde{\boldsymbol{A}}=\boldsymbol{A}\boldsymbol{\mathcal{G}}.\label{eq:preconditioned-space}
\end{equation}
We will address how to choose $\boldsymbol{b}_{i}$ momentarily. Let
$\boldsymbol{V}_{i-1}$ be composed of the computed null-space basis
vectors up to the $(i-1)$st iteration, starting from $\boldsymbol{V}_{0}=\emptyset$.
Then, 
\begin{equation}
\boldsymbol{x}_{i}=(\boldsymbol{I}-\boldsymbol{V}_{i-1}\boldsymbol{V}_{i-1}^{T})\boldsymbol{\mathcal{G}}\boldsymbol{y}_{i}\label{eq:orthogonal-null-vector}
\end{equation}
is a new orthogonal basis vector of $\mathcal{N}(\boldsymbol{A})$,
and we append $\boldsymbol{v}_{i}=\boldsymbol{x}_{i}/\left\Vert \boldsymbol{x}_{i}\right\Vert $
to $\boldsymbol{V}_{i-1}$ to obtain $\boldsymbol{V}_{i}$. To find
a complete set of null-space vectors, we can repeat the outer-most
iteration until the \emph{null-space residual} $\boldsymbol{A}\boldsymbol{v}_{i+1}$
has a large norm relative to the norms of $\boldsymbol{A}$ and $\boldsymbol{v}_{i+1}$,
such as $\left\Vert \boldsymbol{A}\boldsymbol{v}_{i+1}\right\Vert _{1}\gg\epsilon_{\text{mach}}\left\Vert \boldsymbol{A}\right\Vert _{1}\left\Vert \boldsymbol{v}_{i+1}\right\Vert _{1}$.
Here, we use the 1-norm (instead of the 2-norm) for computational
efficiency; alternatively, the $\infty$-norm may also be used.  

For the above procedure to achieve machine precision, the most critical
parts are the choices of $\boldsymbol{\mathcal{G}}$ and $\boldsymbol{b}_{i}$.
In terms of $\boldsymbol{\mathcal{G}}$, we use HIFIR with QRCP without
truncation for the final Schur complement (to avoid division by zero,
it suffices to replace any zero diagonal in $\boldsymbol{R}$ from
QRCP with $\epsilon_{\text{mach}}r_{11}$). Disabling truncation may
sound counterintuitive, because it leads to an ill-conditioned HIF,
and both HIFIR and FGMRES may diverge. Due to submultiplicity of norms,
a very large $\left\Vert \boldsymbol{y}_{i}\right\Vert $ and a small
$\left\Vert \boldsymbol{A}\boldsymbol{\mathcal{G}}\boldsymbol{y}_{i}\right\Vert $
imply that $\boldsymbol{y}_{i}$ is approximately in $\mathcal{N}(\boldsymbol{A}\boldsymbol{\mathcal{G}})$
and $\boldsymbol{x}_{i}$ is approximately in $\mathcal{N}(\boldsymbol{A})$.
Hence, our goal is indeed to make FGMRES diverge as quickly as possible.
  In terms of $\boldsymbol{b}_{i}$, we start with a set of orthonormal
vectors $\{\boldsymbol{q}_{i}\}$. For each $\boldsymbol{q}_{i}$,
if the final Schur complement $\boldsymbol{S}_{m}$ in HIF is ill-conditioned
(e.g., $\kappa(\boldsymbol{S}_{m})>10^{10}$), we use $\boldsymbol{q}_{i}$
as $\boldsymbol{b}_{i}$ directly; otherwise, we apply a few iterations
of HIFIR with a large $\beta_{U}$ ($\beta_{U}=10^{8}$) on $\boldsymbol{q}_{i}$
to obtain $\boldsymbol{b}_{i}$ to accelerate the convergence (or
divergence) of the null-space vector.

This triple-loop algorithm is highly nonlinear, so it is difficult
to analyze the whole process rigorously. To develop a heuristic justification,
let us assume that $\boldsymbol{A}$ is nonsingular to machine precision
(i.e., $\kappa(\boldsymbol{A})=\text{\ensuremath{\sigma_{1}}(\ensuremath{\boldsymbol{A}})/\ensuremath{\sigma_{n}}(\ensuremath{\boldsymbol{A}})=}\mathcal{O}(1/\epsilon_{\text{mach}})$)
with numerical null space $\tilde{\mathcal{N}}(\boldsymbol{A})$ and
$\boldsymbol{A}\boldsymbol{G}$ is nearly symmetric. In this case,
if $\left\Vert \boldsymbol{x}_{j}\right\Vert \gg\left\Vert \boldsymbol{b}\right\Vert $
in \eqref{eq:fixed-point-iteration}, HIFIR is analogous to power
iterations on $\boldsymbol{I}-\boldsymbol{G}\boldsymbol{A}$. Hence,
applying HIFIR on $\boldsymbol{q}_{i}$ makes the solution vector
approximately parallel to a dominant eigenvector of $\boldsymbol{G}\boldsymbol{A}$,
or a least-dominant eigenvector of $\boldsymbol{A}^{-1}\boldsymbol{G}^{-1}$.
If $\boldsymbol{G}\approx\boldsymbol{A}^{-1}$, $\boldsymbol{q}_{i}$
is approximately in $\tilde{\mathcal{N}}(\boldsymbol{A})$, and $\mathcal{K}_{k}(\boldsymbol{A},\boldsymbol{v},\boldsymbol{\mathcal{G}}_{k})$
with HIFIR as a variable preconditioner would be nearly parallel to
$\tilde{\mathcal{N}}(\boldsymbol{A})$. As a result, the computed
null-space vector will be less ``noisy'' than using a random $\boldsymbol{b}_{i}$.
However, if $\kappa(\boldsymbol{G})\gg\mathcal{O}(1/\epsilon_{\text{mach}})$
for $\boldsymbol{G}$ in \eqref{eq:mlilu-pseudoinverse}, which may
happen when $\kappa(\boldsymbol{S}_{m})\gg\mathcal{O}(1/\epsilon_{\text{mach}})$,
then the HIFIR-preconditioned GMRES may converge rapidly even with
a random $\boldsymbol{b}_{i}$. In this case, we do not apply iterative
refinement on $\boldsymbol{q}_{i}$ to avoid over-solve. Note that
our algorithm computes $\boldsymbol{b}_{i}$ starting from orthogonal
vectors, because if the $\boldsymbol{b}_{i}$ were (nearly) parallel
to begin with, the computed null-space vectors would have been (nearly)
parallel, and the orthogonalization step in \eqref{eq:orthogonal-null-vector}
would have been dominated by rounding (or cancellation) errors.

We note three algorithmic details. First, we need to introduce new
stopping criteria in FGMRES. We terminate FGMRES when $\boldsymbol{A}\boldsymbol{x}_{k}$
has reached the desired threshold, or $\boldsymbol{A}\boldsymbol{x}_{k}$
has stagnated after $\left\Vert \boldsymbol{A}\boldsymbol{x}_{k}\right\Vert $
is small enough (e.g., $\left\Vert \boldsymbol{A}\boldsymbol{x}_{k}\right\Vert _{1}<10^{-11}\left\Vert \boldsymbol{A}\right\Vert _{1}\left\Vert \boldsymbol{x}_{k}\right\Vert _{1}$,
where the 1-norm is used for computational efficiency). Note that
$\boldsymbol{A}\boldsymbol{x}_{k}$ is not a side product in FGMRES
and must be computed explicitly. For efficiency, we monitor the Hessenberg
matrix (i.e., $\boldsymbol{H}_{k}$) in the Arnoldi process as we
do for $\boldsymbol{S}_{m}$ in HIF, and we compute $\boldsymbol{A}\boldsymbol{x}_{k}$
only if $\boldsymbol{S}_{m}$ is numerically singular or $\boldsymbol{H}_{k}$
is ill-conditioned (e.g., $\kappa(\boldsymbol{H}_{k})>10^{6}$). Second,
we change the Gram--Schmidt orthogonalization in the Arnoldi process
to use Householder QR as described in \cite[Algorithm 6.10]{saad2003iterative},
since the loss of orthogonality from Gram--Schmidt can prevent the
solution from reaching machine precision. Third, for the same reason,
we use Householder QR in \eqref{eq:orthogonal-null-vector}. We will
report numerical results in section~\ref{subsec:Numerical-comparison-of}. 

\subsection{\label{subsec:pipit}Applications to pseudoinverse solutions of PDEs}

FGMRES+HIFIR as described in section~\ref{subsec:finding-null-spaces}
cannot solve inconsistent systems accurately because the solution
would, in general, converge (or diverge) to a null-space vector of
$\boldsymbol{A}$. We now describe an algorithm for computing the
pseudoinverse solution by combining the techniques in sections~\ref{sec:HIF}
and \ref{sec:HIFIR}, especially for large-scale singular systems
from well-posed PDE discretizations. Such systems often arise from
the Poisson equation with periodic or Neumann boundary conditions
\cite{LeVeque07FDM,bochev2005finite}, incompressible Navier--Stokes
(INS) equations with ``do-nothing'' boundary conditions on pressure
\cite{elman2014finite}, quasi-electrostatic problems with divergence-free
fields \cite{reddy2010finite}, or mechanical systems invariant of
translation and rotation \cite{Dhondt2004}.

\subsubsection{\label{subsec:PIPIT}Pseudoinverse solution via preconditioned iterative
method}

To obtain the pseudoinverse solutions, we propose a solver called
\emph{PIPIT}, which stands for \emph{pseudoinverse solver via preconditioned
iterations}. The algorithm proceeds as follows:
\begin{enumerate}
\item compute orthonormal basis vectors of $\mathcal{N}(\boldsymbol{A}^{T})$
iteratively using FGMERS+HIFIR on $\boldsymbol{A}^{T}$ as described
in section~\ref{subsec:finding-null-spaces}, and store the vectors
in $\boldsymbol{U}$;
\item solve $\boldsymbol{A}\boldsymbol{x}_{\text{LS}}=(\boldsymbol{I}-\boldsymbol{U}\boldsymbol{U}^{T})\boldsymbol{b}$
using GMRES+HIF as described in section~\ref{sec:HIF};
\item $\boldsymbol{x}_{\text{PI}}=(\boldsymbol{I}-\boldsymbol{V}\boldsymbol{V}^{T})\boldsymbol{x}_{\text{LS}}$,
where $\boldsymbol{V}$ contains the orthonormal basis vectors of
$\mathcal{N}(\boldsymbol{A})$.
\end{enumerate}
In step 3, if $\boldsymbol{V}$ is unknown \emph{a priori}, it can
be computed using FGMERS+HIFIR on $\boldsymbol{A}$.

This three-step algorithm is remarkably simple. It essentially converts
an inconsistent system into a consistent system in the first two steps
and then converts an LS solution from step 2 to the pseudoinverse
solution in step 3. Its logic is the same as using an RRQR to compute
a pseudoinverse solution of an RDLS system with known $\mathcal{N}(\boldsymbol{A})$
(e.g., due to range symmetry). However, PIPIT does not compute the
QR on the full matrix, and it can still work if $\mathcal{N}(\boldsymbol{A})$
is unknown or $\mathcal{N}(\boldsymbol{A})\neq\mathcal{N}(\boldsymbol{A}^{T})$.
It is worth noting that in \cite{brown1997gmres}, Brown and Walker
also suggested converting an inconsistent system into a consistent
system \emph{if possible}. However, such a strategy had not been practical
for large-scale sparse systems due to a lack of efficient algorithms
to compute $\mathcal{N}(\boldsymbol{A}^{T})$. The principal enabler
of PIPIT is FGMERS+HIFIR. It is worth noting that incomplete factorization
is often the most expensive part of PIPIT. A key feature of PIPIT
is that steps 1 and 2 reuse the same HIF of $\boldsymbol{A}$, except
that step 2 applies truncation to the $\boldsymbol{Q}$ and $\boldsymbol{R}$
factors of $\boldsymbol{S}_{m}$, but step 1 does not. If $\boldsymbol{V}$
needs to be computed explicitly in step 3, it would also reuse the
same HIF.

\subsubsection*{Comparison with other iterative pseudoinverse solvers}

In \cite{bjorck1979accelerated}, Bj\"orck and Elfving described
two algorithms for finding the pseudoinverse solution of inconsistent
systems. Both algorithms started by applying a CGLS-type method to
obtain $\boldsymbol{x}_{\text{LS}}$ for inconsistent systems, and
then they applied different CGLS-type methods to solve $\boldsymbol{A}\boldsymbol{x}_{\text{PI}}=\boldsymbol{A}\boldsymbol{x}_{\text{LS}}$
\cite[Algorithm 6.1]{bjorck1979accelerated} and $\boldsymbol{A}\boldsymbol{x}_{\text{LS}-\text{PI}}=\boldsymbol{0}$
\cite[Algorithm 6.2]{bjorck1979accelerated}, respectively. Unpreconditioned
LSQR \cite{paige1982lsqr} and LSMR \cite{fong2011lsmr} can solve
for the pseudoinverse solution of inconsistent systems in a single
pass; however, if a preconditioner is used, these methods only produce
a (weighted) LS solution, unless they also employ a second step analogous
to that of Bj\"orck and Elfving. These CGLS-type methods are independent
of the null-space dimension, but they converge slowly due to squaring
the condition number in the normal equation. In contrast, PIPIT does
not square the condition number, but it works better when the null
space is low dimensional. In \cite{bochev2005finite}, Bochev and
Lehoucq used the augmented Lagrangian method to solve singular systems
from Galerkin methods for the Poisson equation with Neumann boundary
conditions. They assumed \emph{a priori} knowledge about the null
spaces and did not address the issue of preconditioning. 

\subsubsection{\label{subsec:Properties-and-assumptions}Justification of low-dimensional
unknown null spaces}

PIPIT works efficiently only when the dimension of the unknown null
space is small. These assumptions are reasonable for many singular
systems from PDEs, as mentioned earlier. For these applications, if
the numerical discretization is well-posed, the right null space (RNS)
of the coefficient matrix is due to unconstrained modes in the continuum
formulation. Hence, the null-space dimension is a small constant,
independently of the number of unknowns. For the same reason, the
RNS of the coefficient matrices may often be derived \emph{a priori}
from the properties of the continuum PDEs. For example, in the case
of INS equations, the null space corresponds to the constant mode
in pressure \cite{elman2014finite}. However, the left null-space
(LNS) vectors are typically unknown if the system is range asymmetric,
where asymmetry may be due to non-self-adjoint differential operators
(such as advection-diffusion equations \cite{ern2013theory}), non-Galerkin
methods (such as finite differences \cite{LeVeque07FDM} and their
generalizations \cite{Urena20111849}), or some sophisticated boundary
treatments (such as some boundary conditions in INS \cite{Gresho1987}).
Even if the LNS can be computed using application-specific knowledge,
PIPIT offers a more user-friendly approach. In section~\ref{subsec:Numerical-comparison-of-PI},
we will present numerical results of PIPIT for two PDEs, namely advection-diffusion
equations and Navier's equations in linear elasticity.

\section{Numerical experiments and results\label{sec:Numerical-Experiments-and}}

In this section, we report some numerical results by applying HIF+GMRES
and HIFIR+FGMRES to solve singular systems. We have implemented HIFIR
and (F)GMRES in C++ by extending HILUCSI described in \cite{chen2021hilucsi}.
We conducted our tests on a single node of a cluster running CentOS
7.4 with two 2.5GHz 12-core Intel Xeon E5-2680v3 processors and 64GB
of RAM. We compiled HIFIR using the default GCC 4.8.5 compiler on
the system with the optimization option -O3. For accurate timing,
both turbo and power-saving modes were turned off for the processors.

\subsection{\label{subsec:Solution-of-consistent}Solution of consistent systems}

As a verification of the analysis for consistent systems in section~\ref{sec:HIF},
Table~\ref{tab:Examples-consistent-systems} compares HILUCSI and
HIF as the right preconditioner in GMRES(30) for four consistent systems
from the SuiteSparse Matrix Collection (aka the UFL Matrix Collection)
\cite{DavisHu11UFSPC}. We list the dimensions of the numerical null
spaces. Although \textsf{\small{}coater2 }is small, it is challenging
due to structural singularity (with 64 empty rows) and ill-conditioning.
For \textsf{shyy161}, the original system is consistent, so we used
its right-hand side RHS directly; for the others, the right-hand-side
vectors were either missing or inconsistent, so we constructed $\boldsymbol{b}$
as the row sums of $\boldsymbol{A}$ (i.e., $\boldsymbol{b}=\boldsymbol{A}\boldsymbol{1}$).
We used double precision for all the computations. Without dropping
in HIF, the relative residual of GMRES was less than $10^{-11}$ after
a single iteration for well-conditioned systems, confirming that HIF
is $\epsilon$-accurate. We then enabled dropping in the MLILU portion.
Table~\ref{tab:Examples-consistent-systems} reports the statistics
of the final Schur complement $\boldsymbol{S}_{m}$, the numbers of
GMRES iterations, and the relative residual. In the case of \textsf{shyy161},
for which $\boldsymbol{S}_{m}$ is singular, GMRES with HIF reached
about $10^{-13}$, but GMRES with HILUCSI stagnated with a relative
residual close to 1 (to the ninth decimal place) after 30 iterations.
Note that even though $\boldsymbol{S}_{m}$ is nonsingular for \textsf{bcsstk35},
HIF still improved the precision of the solution over HILUCSI by a
factor of 30 due to the better stability of QRCP over LU factorization.
As a reference, Figures~\ref{fig:bcsstk35} and \ref{fig:shyy161}
compare the convergence history of these preconditioned GMRES with
that of unpreconditioned GMRES for \textsf{bcsstk35} and \textsf{shyy161}
up to 1,000 iterations, respectively. Without preconditioning, GMRES(30)
only reached a relative residual of around $10^{-5}$ after 1,000
iterations.

\begin{table}
\caption{\label{tab:Examples-consistent-systems}Example results of HILUCSI
and HIF preconditioned GMRES(30) for consistent systems, with fill-in
ratio $\alpha=10$ and droptol $\sigma=10^{-4}$. $\text{d}(\mathcal{N})$
denotes the null-space dimension. \textsf{invextr1 }is shorthand for
\textsf{invextr1\_new}. The best results are in bold. For \textsf{coater2},
truncated QRCP could only achieve roughly $\sqrt{\epsilon_{\text{mach}}}$,
so the result with HIF was competitive. GMRES(30)+HILUCSI stagnated
for \textsf{shyy161} with a relative residual close to 1 (0.99999999954).}

\centering{}{\small{}}%
\begin{tabular}{c|c|c|c|c|c|c|c}
\hline 
\multirow{2}{*}{{\small{}Case ID}} & \multicolumn{3}{c|}{{\small{}Matrix info}} & \multicolumn{2}{c|}{{\small{}$\boldsymbol{S}_{m}$ in HIF}} & \multicolumn{2}{c}{{\small{}$\Vert\boldsymbol{r}\Vert$/$\Vert\boldsymbol{b}\Vert$}}\tabularnewline
\cline{2-8} \cline{3-8} \cline{4-8} \cline{5-8} \cline{6-8} \cline{7-8} \cline{8-8} 
 & \multirow{1}{*}{{\small{}$n$}} & \multirow{1}{*}{{\small{}nnz}} & \multirow{1}{*}{{\small{}$\text{d}(\mathcal{N})$}} & {\small{}$n$} & {\small{}$\text{d}(\mathcal{N})$} & {\small{}HILUCSI} & {\small{}HIF}\tabularnewline
\hline 
\textsf{\small{}coater2} & {\small{}9.540} & {\small{}207,308} & $\approx114$ & {\small{}3,411} & {\small{}106} & {\small{}NaN} & \textbf{\small{}1.9e-07}\tabularnewline
\hline 
\textsf{\small{}bcsstk35} & {\small{}30,237} & {\small{}1,450,163} & {\small{}6} & {\small{}1,004} & {\small{}0} & {\small{}2.1e-10} & \textbf{\small{}7.1e-12}\tabularnewline
\hline 
\textsf{\small{}invextr1} & {\small{}30,412} & {\small{}1,793,881} & {\small{}$\approx$2,910} & {\small{}3,273} & {\small{}0} & \textbf{\small{}2.3e-16} & {\small{}3.6e-16}\tabularnewline
\hline 
\textsf{\small{}shyy161} & {\small{}76,480} & {\small{}329,762} & {\small{}$\gtrsim50$} & {\small{}9,394} & {\small{}48} & {\small{}$\sim1$} & \textbf{\small{}1.1e-13}\tabularnewline
\hline 
\end{tabular}{\small\par}
\end{table}

\begin{figure}
\begin{minipage}[t]{0.48\columnwidth}%
\begin{center}
\includegraphics[width=0.9\columnwidth]{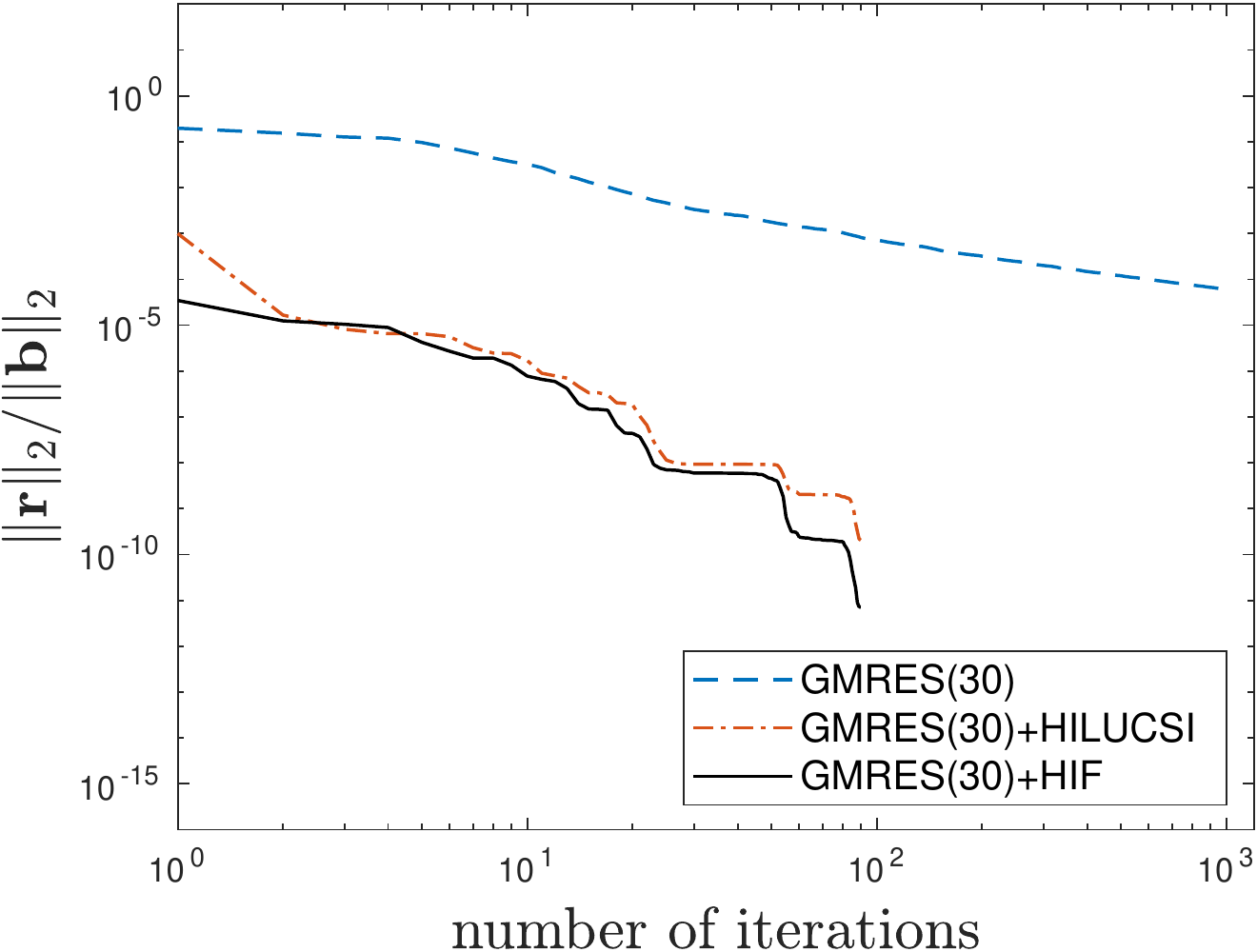}\caption{\label{fig:bcsstk35}Comparison of convergence history of preconditioned
GMRES(30) for \textsf{bcsstk35}{\footnotesize{}.}}
\par\end{center}%
\end{minipage}\hfill%
\begin{minipage}[t]{0.48\columnwidth}%
\begin{center}
\includegraphics[width=0.9\columnwidth]{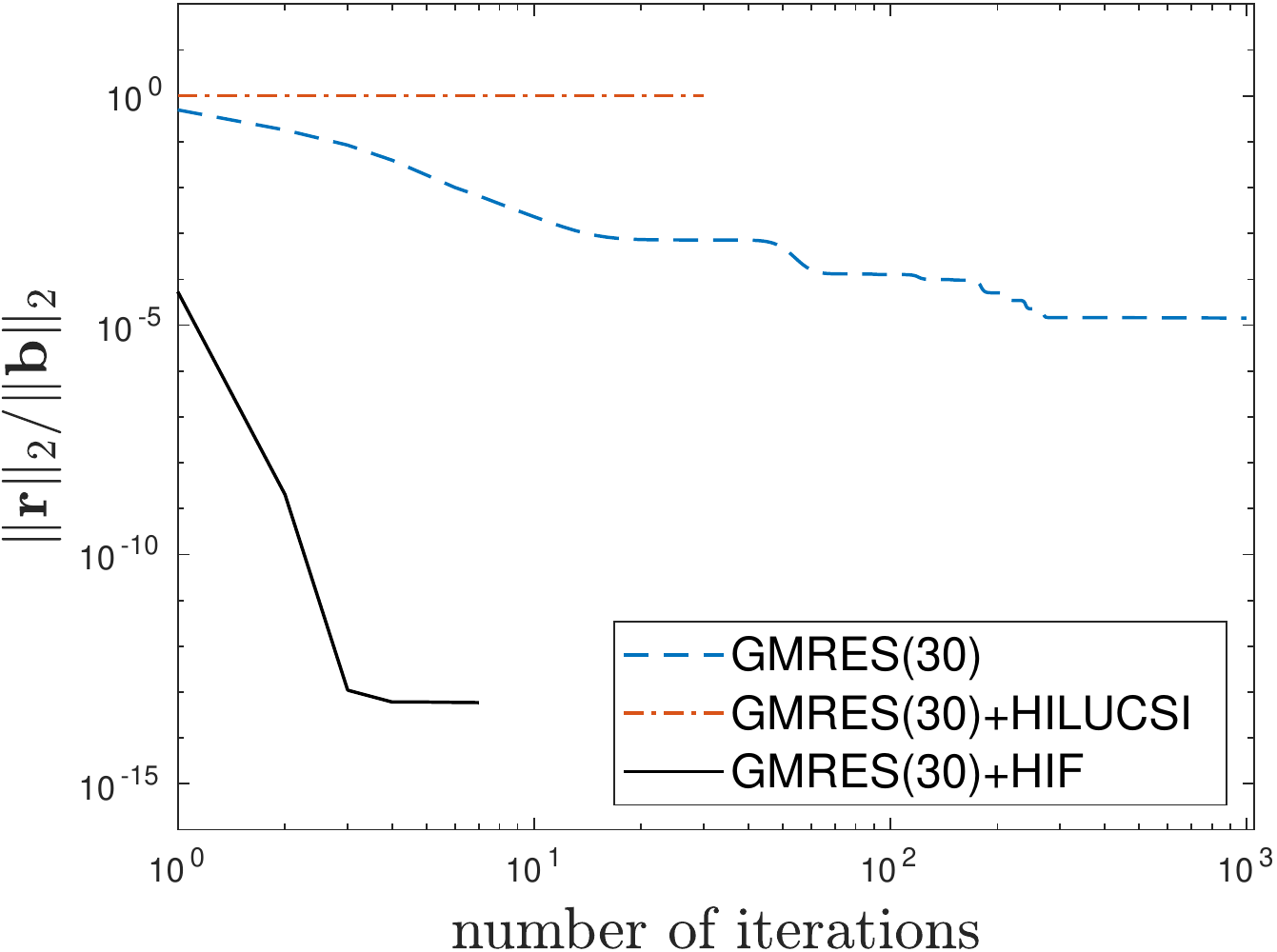}\caption{\label{fig:shyy161}Comparison of convergence history of preconditioned
GMRES(30) for \textsf{shyy161}{\footnotesize{}.}}
\par\end{center}%
\end{minipage}
\end{figure}

\subsection{\label{subsec:Numerical-comparison-of}Numerical comparison of null-space
computations}

To assess the accuracy and efficiency of HIFIR, we compare it with
sparse SVD \cite{baglama2005augmented} as implemented in the \textsf{svds}
function in MATLAB \cite{MATLAB2020a} as well as RIF-preconditioned
LSMR \cite{fong2011lsmr,benzi2003robust}. The logic of computing
null spaces using LSMR+RIF follows that of \textsf{MLSQRnull} \cite[Section 5.3.3]{choi2006iterative}:
Given $\boldsymbol{R}$ from RIF such that $\boldsymbol{A}^{T}\boldsymbol{A}\approx\boldsymbol{R}^{T}\boldsymbol{R}$
and a random $\boldsymbol{b}$ that is orthogonal to the previously
computed LNS vectors, $\boldsymbol{b}-\boldsymbol{A}\boldsymbol{x}_{\text{LS}}$
is a new LNS vector of $\boldsymbol{A}$, where $\boldsymbol{x}_{\text{LS}}=\boldsymbol{R}^{-1}\boldsymbol{y}_{\text{PI}}$
with $\boldsymbol{y}_{\text{PI}}\approx\left(\boldsymbol{A}\boldsymbol{R}^{-1}\right)^{+}\boldsymbol{b}$.
We used the suggested control parameters for LSMR to achieve maximal
precision. For RIF, we used the FORTRAN implementation of its implicit
left-looking algorithm \cite{scott2016preconditioning} with drop
tolerance 0.1 as in \cite{gould2017state}, and compiled it using
gfortran 4.8.5 with the -O3 option.

We first compared the algorithms for problems with a one-dimensional
null space. To this end, we used the MATLAB function \textsf{gallery('neumann',
$n$) }to construct three systems with $n=$$64^{2}$, $256^{2}$,
and $1024^{2}$, which correspond to the finite-difference discretization
of the two-dimensional Poisson equation with Neumann boundary conditions
on uniform grids of the corresponding sizes, respectively. The RNS
is known \emph{a priori }in that it corresponds to a constant mode,
and the LNS for this simple example also has a simple pattern. Nevertheless,
we computed both LNS and RNS in our tests. For HIFIR the LNS is computed
by applying the algorithm on $\boldsymbol{A}^{T}$ with $\boldsymbol{G}^{T}$
as the preconditioner. Note that LSMR requires two different preconditioners
for LNS and RNS, since $\boldsymbol{A}^{T}\boldsymbol{A}\neq\boldsymbol{A}\boldsymbol{A}^{T}$
in general. For \textsf{svds}, we used the command \textsf{{[}u,s,v{]}=svds(A,1,'smallest')}
to compute the smallest singular value and both of its corresponding
singular vectors in a single step. We conducted our comparison using
a single core on a system with 2.1GHz Intel Xeon E7-8870 v3 CPUs and
3TB of RAM.

Table~\ref{tab:neumann-null-spaces} summarizes the relative null-space
residuals as well as the runtimes of the methods. In terms of accuracy,
our proposed approach, namely FGMRES+HIFIR, consistently achieved
machine precision, while both LSMR+RIF and \textsf{svds }suffered
from severe loss of precision as the problem sizes increased. As a
result, for the $1024^{2}$ cases, our approach was about two to three
orders of magnitude more accurate than \textsf{svds} and LSMR, respectively.
Our approach was also one to two orders faster than the prior state
of the art, and its computational cost also grew at a much slower
pace. Its accuracy and efficiency are because (1) HIFIR is effective
at preconditioning singular matrices, and (2) FGMRES does not square
the condition number (unlike LSMR). In terms of memory, HIFIR used
107MB and 1.4GB for the medium and large Neumann systems, respectively,
which were roughly proportional to the input sizes. In contrast, RIF
used 110MB and 0.6GB for the two systems, which grew sublinearly.
This sublinear complexity indicates that RIF preserved fewer nonzeros
per row as the problem size grew. Since these systems have a constant
number of nonzeros per row independently of the problem sizes, such
increased dropping per row for larger systems limited the effectiveness
of RIF for large systems. In comparison, \textsf{svds }used 5.4GB
and 208GB for the two cases, respectively, about two orders of magnitude
larger than HIFIR.

\begin{table}
\caption{\label{tab:neumann-null-spaces}Example results of computing null
spaces. $\boldsymbol{v}$ denotes a unit-length null-space vector.
Relative residuals are in 2-norm, and the unit is machine epsilon
(i.e., $\epsilon_{\text{mach}}\approx2.22e-16$). The best results
are in bold.}

\centering{}{\small{}}%
\begin{tabular}{c|c|c|c|c|c|c|c|c|c}
\hline 
\multirow{3}{*}{{\small{}Case ID}} & \multirow{3}{*}{{\small{}$n$}} & \multicolumn{3}{c|}{{\small{}$\Vert\boldsymbol{A}\boldsymbol{v}\Vert/\Vert\boldsymbol{A}\Vert$
$(\epsilon_{\text{mach}})$}} & \multicolumn{5}{c}{{\small{}Runtime (s)}}\tabularnewline
\cline{3-10} \cline{4-10} \cline{5-10} \cline{6-10} \cline{7-10} \cline{8-10} \cline{9-10} \cline{10-10} 
 &  & {\scriptsize{}FGMRES} & {\scriptsize{}LSMR} & \multirow{2}{*}{\textsf{\small{}svds}} & \multicolumn{2}{c|}{{\scriptsize{}FGMRES+HIFIR}} & \multicolumn{2}{c|}{{\scriptsize{}LSMR+RIF}} & \multirow{2}{*}{\textsf{\small{}svds}}\tabularnewline
\cline{6-9} \cline{7-9} \cline{8-9} \cline{9-9} 
 &  & {\scriptsize{}+HIFIR} & {\scriptsize{}+RIF} &  & {\small{}fac.} & {\small{}sol.} & {\small{}fac.} & {\small{}sol.} & \tabularnewline
\hline 
{\small{}small-RNS} & \multirow{2}{*}{{\small{}$64^{2}$}} & \textbf{\small{}0.33} & {\small{}7.95} & {\small{}5.44} & \multirow{2}{*}{\textbf{\small{}0.07}} & \textbf{\small{}0.096} & {\small{}0.13} & {\small{}0.46} & \multirow{2}{*}{{\small{}0.43}}\tabularnewline
\cline{1-1} \cline{3-5} \cline{4-5} \cline{5-5} \cline{7-9} \cline{8-9} \cline{9-9} 
{\small{}small-LNS} &  & \textbf{\small{}0.35} & {\small{}8.97} & {\small{}4.37} &  & \textbf{\small{}0.091} & {\small{}0.13} & {\small{}0.58} & \tabularnewline
\hline 
{\small{}mid-RNS} & \multirow{2}{*}{{\small{}$256^{2}$}} & \textbf{\small{}0.38} & {\small{}54.3} & {\small{}9.92} & \multirow{2}{*}{\textbf{\small{}0.89}} & \textbf{\small{}2.61} & {\small{}3.34} & {\small{}41} & \multirow{2}{*}{{\small{}26.59}}\tabularnewline
\cline{1-1} \cline{3-5} \cline{4-5} \cline{5-5} \cline{7-9} \cline{8-9} \cline{9-9} 
{\small{}mid-LNS} &  & \textbf{\small{}0.36} & {\small{}51.6} & {\small{}8.36} &  & \textbf{\small{}2.39} & {\small{}3.46} & {\small{}40} & \tabularnewline
\hline 
{\small{}large-RNS} & \multirow{2}{*}{{\small{}$1024^{2}$}} & \textbf{\small{}0.65} & {\small{}1.2e3} & {\small{}23.2} & \multirow{2}{*}{\textbf{\small{}11.6}} & \textbf{\small{}182} & {\small{}236} & {\small{}9.8e3} & \multirow{2}{*}{{\small{}5.4e3}}\tabularnewline
\cline{1-1} \cline{3-5} \cline{4-5} \cline{5-5} \cline{7-9} \cline{8-9} \cline{9-9} 
{\small{}large-LNS} &  & \textbf{\small{}0.54} & {\small{}775} & {\small{}52.7} &  & \textbf{\small{}167} & {\small{}236} & {\small{}9.8e3} & \tabularnewline
\hline 
\end{tabular}{\small\par}
\end{table}

To gain more insights, we compared unpreconditioned GMRES(30), GMRES(30)
with HIF, FGMRES(30) with HIFIR, unpreconditioned LSMR, and LSMR with
RIF. We started with the same initial vectors for GMRES and FGMRES,
and we limited the number of iterations of LSMR to $10^{4}$ as in
\cite{gould2017state}. Figures~\ref{fig:null-space} and \ref{fig:null-space2}
show the convergence history of these methods for mid-LNS and large-LNS,
respectively. In the figures, the $x$-axis corresponds to the number
of matrix-vector products, which is equal to the total number of inner
iterations for HIFIR and is equal to two times the number of iterations
for LSMR. Note that the $x$-axis is in logarithmic scales because
FGMRES+HIFIR would have been invisible if the linear scale were used.
It can be seen that unpreconditioned GMRES stagnated at about $10^{-5}$
for both cases. GMRES+HIF reached machine precision for mid-LNS after
some bumps due to restart, but it stagnated at about $10^{-8}$ for
large-LNS. FGMRES+HIFIR achieved machine precision consistently, at
the cost of more matrix-vector products than GMRES+HIF to reach higher
precision ($10^{-16}$). In terms of LSMR, RIF accelerated it by about
an order of magnitude, but LSMR+RIF could not reach $10^{-8}$ for
large-LNS after $10^{4}$ iterations.

\begin{figure}
\begin{minipage}[t]{0.48\columnwidth}%
\begin{center}
\includegraphics[width=0.9\columnwidth]{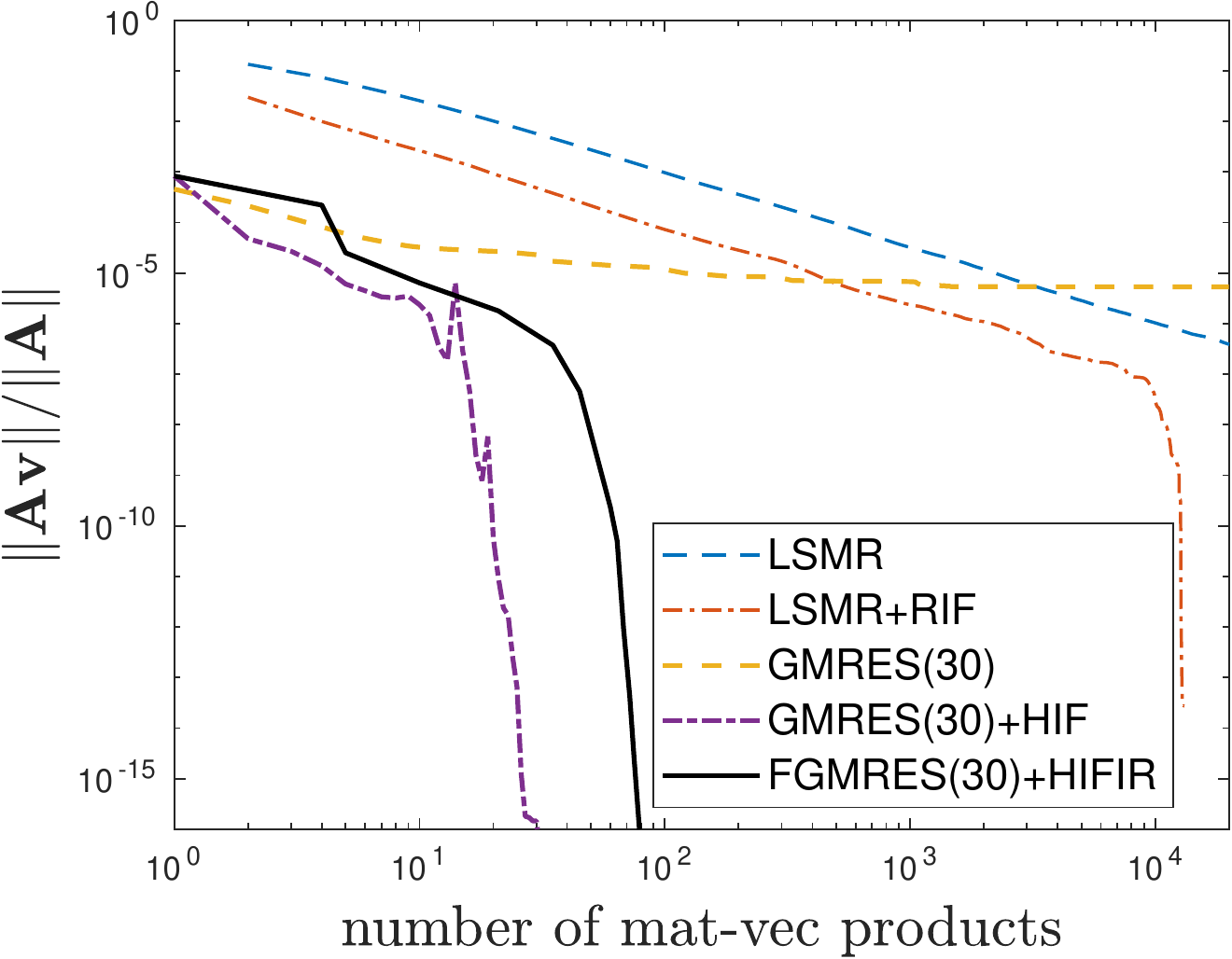}\caption{\label{fig:null-space}Comparison of convergence history for mid-LNS
in 2-norm.}
\par\end{center}%
\end{minipage}\hfill%
\begin{minipage}[t]{0.48\columnwidth}%
\begin{center}
\includegraphics[width=0.9\columnwidth]{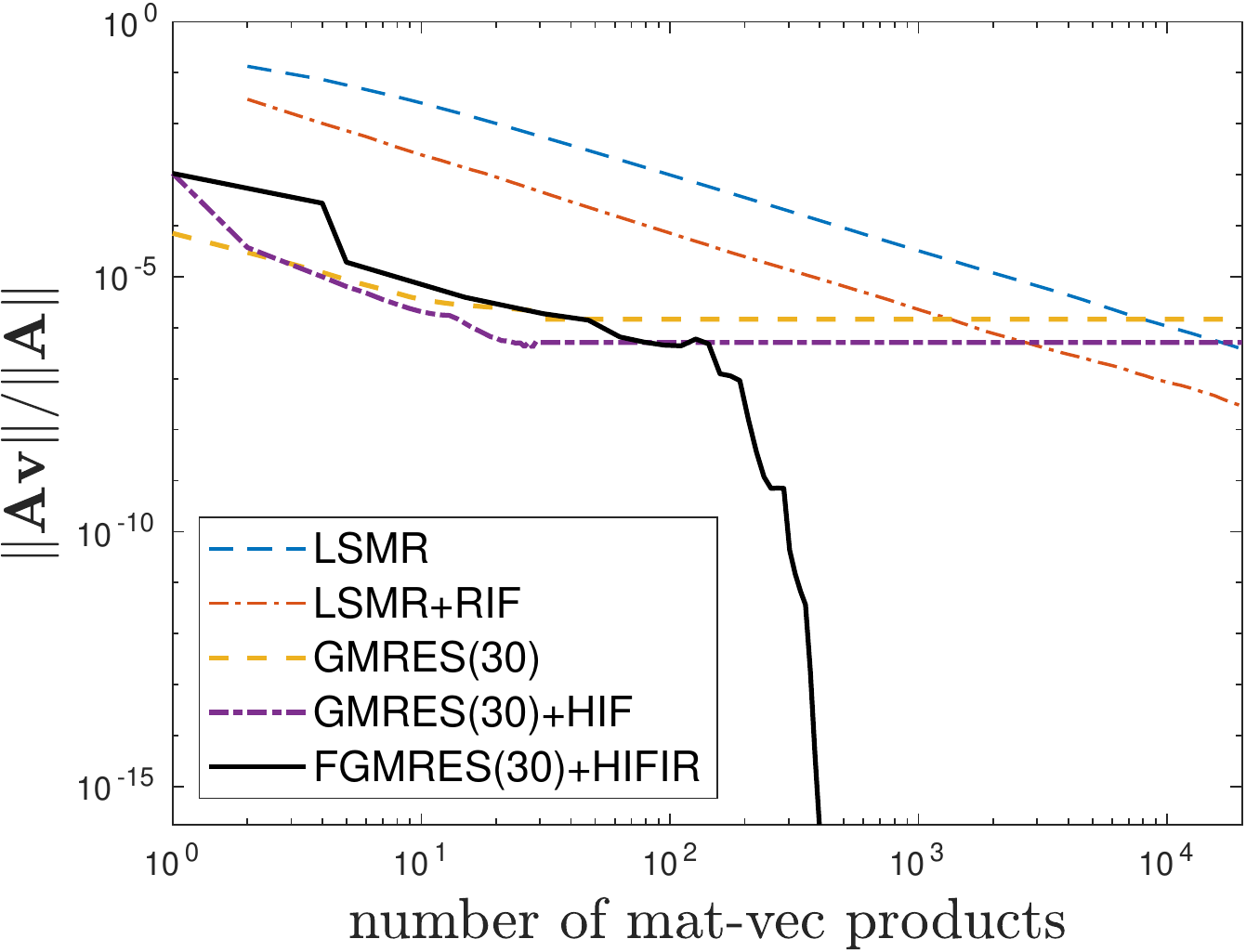}\caption{\label{fig:null-space2}Comparison of convergence history for large-LNS
in 2-norm.}
\par\end{center}%
\end{minipage}
\end{figure}

To assess the accuracy and efficiency of the methods for multidimensional
null spaces, we computed 20 null-space vectors of \textsf{invextr1}
and \textsf{shyy161}. Table~\ref{tab:multidim-null-spaces} reports
the null-space residuals of the first, 10th, and 20th null-space vectors,
as well as the runtimes. For FMGRES+HIFIR, the preprocessing and post-processing
times to obtain $\{\boldsymbol{b}_{i}\}$ and to orthogonalize $\{\boldsymbol{v}_{i}\}$
accounted for less than 1\% of the total times, so we omitted them.
For \textsf{invextr1}, LSMR+RIF produced a null-space residual of
about 0.02 for all the null-space vectors, so we considered it as
a failure;\footnote{Unpreconditioned LSMR also failed for \textsf{invextr1} by producing
a null-space residual $>10^{-5}$ after $10^{4}$ iterations.} compared to \textsf{svds}, FMGRES+HIFIR was about 10 to $10^{4}$
times more accurate and about 30\% faster. For \textsf{shyy161}, \textsf{svds}
produced a relative residual of about 0.2 starting from the third
smallest singular value with a warning message about ill-conditioning,
so we also consider it as a failure; compared to LSMR+RIF, FMGRES+HIFIR
was more accurate by a factor of $10^{5}$ and was faster by a factor
of five.

\begin{table}
\caption{\label{tab:multidim-null-spaces}Example results of computing multi-dimensional
null-space vectors. The $\boldsymbol{v}_{i}$ denote the $i$th unit-length
null-space vector. LSMR and \textsf{svds }failed for \textsf{invextr1}
and \textsf{shyy161}, respectively. The best results are in bold.}

\centering{}{\small{}}%
\begin{tabular}{c|c|c|c|c|c|c|c|c}
\hline 
\multirow{3}{*}{{\small{}Case ID}} & \multicolumn{3}{c|}{{\small{}$\Vert\boldsymbol{A}\boldsymbol{v}_{i}\Vert/\Vert\boldsymbol{A}\Vert$}} & \multicolumn{5}{c}{{\small{}Runtime (s)}}\tabularnewline
\cline{2-9} \cline{3-9} \cline{4-9} \cline{5-9} \cline{6-9} \cline{7-9} \cline{8-9} \cline{9-9} 
 & {\small{}FGMRES} & {\small{}LSMR} & \multirow{2}{*}{\textsf{\small{}svds}} & \multicolumn{2}{c|}{{\small{}HIFIR}} & \multicolumn{2}{c|}{{\small{}RIF}} & \multirow{2}{*}{\textsf{\small{}svds}}\tabularnewline
\cline{5-8} \cline{6-8} \cline{7-8} \cline{8-8} 
 & {\small{}+HIFIR} & {\small{}+RIF} &  & {\small{}fac.} & {\small{}sol.} & {\small{}fac.} & {\small{}sol.} & \tabularnewline
\hline 
\textsf{\small{}invextr1 $\boldsymbol{v}{}_{1}$} & \textbf{\small{}1.75e-15} & {\small{}\sout{1.62e-2}} & {\small{}1.65e-11} & \multirow{3}{*}{\textbf{\small{}29.2}} & \multirow{3}{*}{\textbf{\small{}147}} & \multirow{3}{*}{{\small{}\sout{1.5e4}}} & \multirow{3}{*}{{\small{}\sout{2.5e4}}} & \multirow{3}{*}{{\small{}255}}\tabularnewline
\cline{1-4} \cline{2-4} \cline{3-4} \cline{4-4} 
\textsf{\small{}invextr1 $\boldsymbol{v}_{10}$} & \textbf{\small{}1.29e-13} & {\small{}\sout{1.48e-2}} & {\small{}1.61e-11} &  &  &  &  & \tabularnewline
\cline{1-4} \cline{2-4} \cline{3-4} \cline{4-4} 
\textsf{\small{}invextr1 $\boldsymbol{v}_{20}$} & \textbf{\small{}1.36e-12} & {\small{}\sout{1.55e-2}} & {\small{}1.61e-11} &  &  &  &  & \tabularnewline
\hline 
\textsf{\small{}shyy161 $\boldsymbol{v}_{1}$} & \textbf{\small{}3.31e-20} & {\small{}3.01e-9} & {\small{}7.92e-19} & \multirow{3}{*}{{\small{}233}} & \multirow{3}{*}{\textbf{\small{}6.2}} & \multirow{3}{*}{\textbf{\small{}3.33}} & \multirow{3}{*}{{\small{}1.1e3}} & \multirow{3}{*}{{\small{}\sout{5.73}}}\tabularnewline
\cline{1-4} \cline{2-4} \cline{3-4} \cline{4-4} 
\textsf{\small{}shyy161 $\boldsymbol{v}_{10}$} & \textbf{\small{}3.67e-15} & {\small{}3.43e-9} & {\small{}\sout{2.05e-1}} &  &  &  &  & \tabularnewline
\cline{1-4} \cline{2-4} \cline{3-4} \cline{4-4} 
\textsf{\small{}shyy161 $\boldsymbol{v}_{20}$} & \textbf{\small{}4.63e-14} & {\small{}6.79e-9} & {\small{}\sout{2.06e-1}} &  &  &  &  & \tabularnewline
\hline 
\end{tabular}{\small\par}
\end{table}

\subsection{\label{subsec:Numerical-comparison-of-PI}Numerical comparison of
pseudoinverse solutions}

To assess the effectiveness of PIPIT, we solved the two example PDEs
in three dimensions. The first one was the advection-diffusion equation
over $\Omega\subset\mathbb{R}^{3}$ with Neumann boundary conditions
on $\partial\Omega$, i.e., 
\begin{align}
-\Delta u+\boldsymbol{v}\cdot\boldsymbol{\nabla}u & =f\qquad\text{in }\text{\ensuremath{\Omega}},\label{eq:convection-diffusion}\\
\nabla u\cdot\boldsymbol{n} & =g\qquad\text{on }\text{\ensuremath{\partial}\ensuremath{\ensuremath{\Omega}}},
\end{align}
where $u$ is the unknown, $\boldsymbol{v}$ is a divergence-free
velocity field, $f$ is a source term, $\boldsymbol{n}$ is the unit
outward surface normal, and $g$ is out-flux. Since the differential
operator is not self-adjoint, a consistent discretization typically
leads to a range-asymmetric algebraic equation, of which the RNS corresponds
to the constant mode (i.e., $\mathcal{N}(\boldsymbol{A})=\text{span}\{\boldsymbol{1}\}$).
However, the LNS is unknown \emph{a priori} due to range asymmetry.
We are interested in obtaining the pseudoinverse solution, which corresponds
to the LS solution with a zero constant mode. For the advection-diffusion
equation, we used $\boldsymbol{v}=[1,1,1]^{T}$ and the exact solution
$u_{*}=e^{xyz}$ over a unit sphere centered at the origin. We generated
tetrahedral meshes using Gmsh \cite{geuzaine2009gmsh} and discretized
the PDE using the generalized finite difference method \cite{Urena20111849}.
Figure~\ref{fig:ad-solution} shows a cut-off of a coarse mesh and
its solution in the $yz$-plane.

The second example is Navier's equations for linear elasticity over
$\Omega\subset\mathbb{R}^{3}$ with pure Neumann boundary condition
on piecewise smooth $\partial\Omega$, i.e.,
\begin{align}
-\boldsymbol{\nabla}\cdot\boldsymbol{\sigma}(\boldsymbol{u}) & =\boldsymbol{f}\text{\ensuremath{\qquad\text{in }\Omega}},\label{eq:linear-elasticity}\\
\boldsymbol{\sigma}(\boldsymbol{u})\cdot\boldsymbol{n} & =\boldsymbol{h}\text{\ensuremath{\qquad\text{on }\partial\Omega},}
\end{align}
where $\boldsymbol{u}$ is the body displacements, $\boldsymbol{\sigma}=\mu(\boldsymbol{\nabla}\boldsymbol{u}+(\boldsymbol{\nabla}\boldsymbol{u})^{T})+\lambda(\boldsymbol{\nabla}\cdot\boldsymbol{u})\boldsymbol{I}$
is the stress tensor with Lam\'e constants $\mu$ and $\lambda$,
$\boldsymbol{f}$ is the body force, and $\boldsymbol{h}$ is the
surface traction. The continuum formulation has a six-dimensional
null space, which corresponds to the degrees of freedom in rigid-body
motion. We used a similar setting as Example 2 in \cite{kuchta2019singular}:
The domain $\Omega$ is obtained by rotating the box $\mathbb{B}=\left[-\frac{1}{4},\frac{1}{4}\right]\times\left[-\frac{1}{2},\frac{1}{2}\right]\times\left[-\frac{1}{8},\frac{1}{8}\right]$
around the $x$-, $y$-, and $z$-axes by angles $\pi/2$, $\pi/4$,
and $\pi/5$ in that order, followed by a translation of $[0.1,0.2,0.3]^{T}$,
i.e., $\Omega=\boldsymbol{Q}\mathbb{B}+[0.1,0.2,0.3]^{T}$ in matrix
notation, where 
\begin{equation}
\boldsymbol{Q}=\begin{bmatrix}\cos\frac{\pi}{5} & -\sin\frac{\pi}{5}\\
\sin\frac{\pi}{5} & \cos\frac{\pi}{5}\\
 &  & 1
\end{bmatrix}\begin{bmatrix}\sin\frac{\pi}{4} &  & \cos\frac{\pi}{4}\\
 & 1\\
\cos\frac{\pi}{4} &  & -\sin\frac{\pi}{4}
\end{bmatrix}\begin{bmatrix}1\\
 & \cos\frac{\pi}{2} & -\sin\frac{\pi}{2}\\
 & \sin\frac{\pi}{2} & \cos\frac{\pi}{2}
\end{bmatrix}.\label{eq:rotation-matrix}
\end{equation}
We manufactured the body force and surface traction as $\boldsymbol{f}=-\boldsymbol{\nabla}\cdot\boldsymbol{\sigma}(\boldsymbol{u}_{*})$
and $\boldsymbol{h}=\boldsymbol{\sigma}(\boldsymbol{u}_{*})\cdot\boldsymbol{n}$,
respectively, where $\mu=384$ and $\lambda=577$ in $\boldsymbol{\sigma}$
and $\boldsymbol{u}_{*}=\frac{1}{4}(\sin\frac{\pi}{4}x,z^{3},-y)$.
In \cite{kuchta2019singular}, Kuchta, Mardal, and Mortensen projected
off the analytical null-space components from $\boldsymbol{f}$ and
$\boldsymbol{h}$. Since PIPIT solves inconsistent systems directly,
such sophisticated preprocessing is not needed in our setting. We
solved the problem using linear finite elements in FEniCS \cite{alnaes2015fenics}.
Figure~\ref{fig:linear-elasticity-solution} shows a coarse mesh
and the warped geometry with four times the computed displacements.
\begin{figure}
\begin{minipage}[t]{0.48\columnwidth}%
\begin{center}
\includegraphics[height=1.2in]{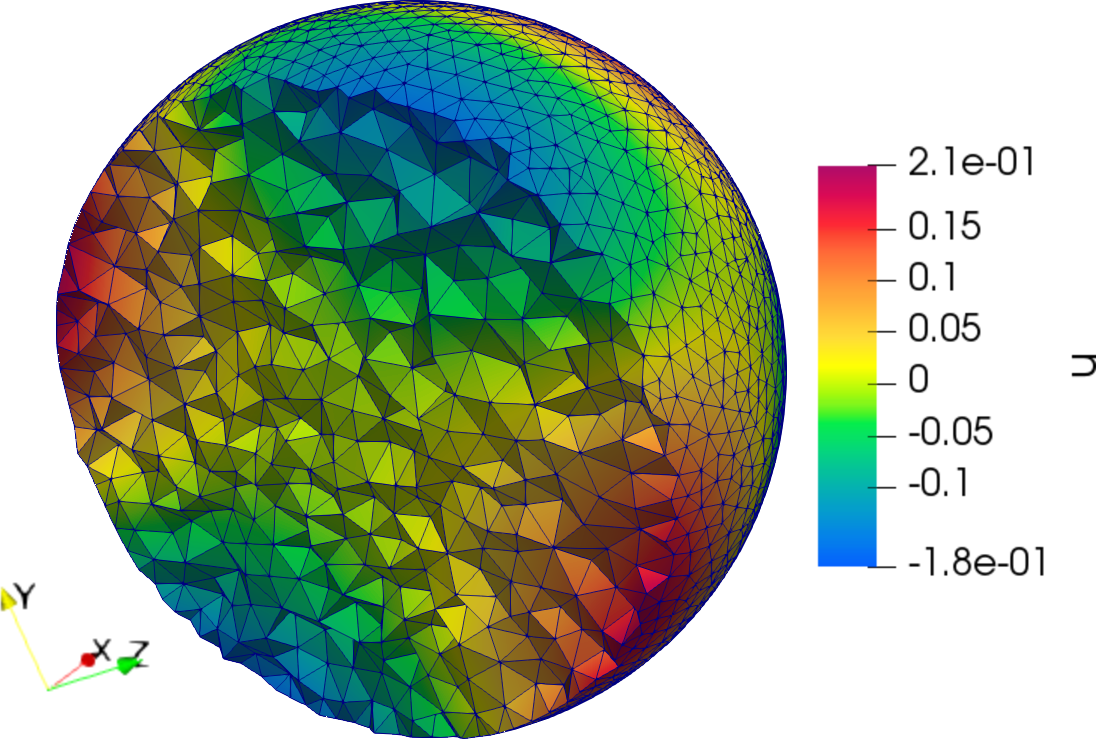}\caption{\label{fig:ad-solution}Example solution of the advection-diffusion
equation with coarse tetrahedral mesh.}
\par\end{center}%
\end{minipage}\hfill%
\begin{minipage}[t]{0.48\columnwidth}%
\begin{center}
\includegraphics[height=1.2in]{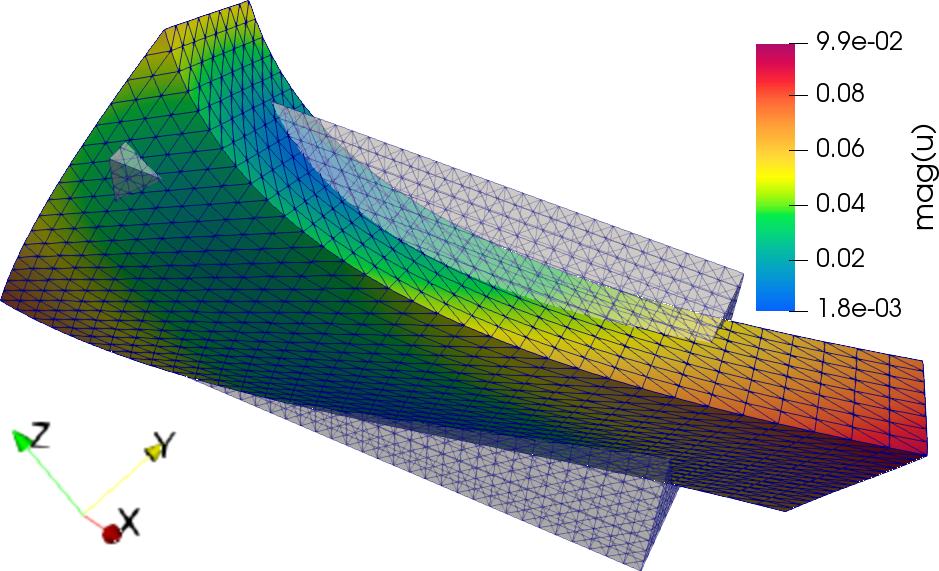}\caption{\label{fig:linear-elasticity-solution}Example coarse mesh and deformed
geometry with exaggerated displacements.}
\par\end{center}%
\end{minipage}
\end{figure}

To study how the problem sizes affect the accuracy of the solution,
for each PDE we generated three matrices, of which the sizes are summarized
in Table~\ref{tab:Examples-pipid}. As points of reference, we compare
our solutions with those of unpreconditioned LSMR and HIF-preconditioned
GMRES. For the latter, we terminate GMRES when the Hessenberg matrix
is moderately ill-conditioned ($\kappa(\boldsymbol{H}_{k})\ge\sqrt{\epsilon_{\text{mach}}}$)
and the residual has stagnated. We refer to this adaptive strategy
as \emph{maximal-precision Brown--Walker} (\emph{MPBW}), in that
its idea is similar to that of Brown and Walker \cite{brown1997gmres},
except that it achieves the maximal precision of GMRES. We limited
the numbers of iterations of (F)GMRES and LSMR to 500 and $10^{4}$,
respectively. We measured the relative residual based on the normal
equation (i.e., $\Vert\boldsymbol{A}^{T}\boldsymbol{r}\Vert_{2}/\Vert\boldsymbol{A}^{T}\boldsymbol{b}\Vert_{2}$).
When LSMR produced a sufficiently accurate residual (in particular
$<\sqrt{\epsilon_{\text{mach}}}$), we report the difference of the
norm of the solution vectors from PIPIT and LSMR. As evident in Table~\ref{tab:Examples-pipid},
PIPIT significantly improved the residuals while minimizing $\Vert\boldsymbol{x}\Vert$
to near machine precision. MPBW did not reach $\sqrt{\epsilon_{\text{mach}}}$
for any of the test cases. More importantly, even when the relative
residuals were at $10^{-5}$, we observed that $\Vert\boldsymbol{x}_{\text{MPBW}}\Vert$
was about $10^{9}$ times larger than $\Vert\boldsymbol{x}_{\text{PIPIT}}\Vert$.
Hence, even though preconditioned GMRES with MPBW may produce an approximate
LS solution, the solution may be unacceptable from the application's
point of view. We omit runtimes since unpreconditioned LSMR converged
too slowly. 

\begin{table}
\caption{\label{tab:Examples-pipid}Example results of pseudoinverse solutions
for inconsistent systems. AD stands for advection-diffusion, and LE
stands for linear elasticity. The best results are in bold; `-' indicates
ineligibility of $\boldsymbol{x}_{\text{LSMR}}$ for comparison.}

\centering{}{\small{}}%
\begin{tabular}{c|c|c|c|c|c|c}
\hline 
\multirow{2}{*}{{\small{}Case ID}} & \multicolumn{2}{c|}{{\small{}Matrix info}} & \multicolumn{3}{c|}{{\small{}$\Vert\boldsymbol{A}^{T}\boldsymbol{r}\Vert/\Vert\boldsymbol{A}^{T}\boldsymbol{b}\Vert$}} & \multicolumn{1}{c}{{\small{}$\Vert\boldsymbol{x}_{\text{PIPIT}}\Vert/$}}\tabularnewline
\cline{2-6} \cline{3-6} \cline{4-6} \cline{5-6} \cline{6-6} 
 & {\small{}$n$} & {\small{}nnz} & {\small{}PIPIT} & {\small{}LSMR} & {\small{}MPBW} & {\small{}$\Vert\boldsymbol{x}_{\text{LSMR}}\Vert-1$}\tabularnewline
\hline 
{\small{}AD-coarse} & {\small{}9,807} & {\small{}640,238} & \textbf{\small{}4.90e-15} & {\small{}2.60e-14} & 8.21e-02 & {\small{}2.9e-15}\tabularnewline
\hline 
{\small{}AD-mid} & {\small{}66,877} & {\small{}3,954,678} & \textbf{\small{}4.21e-15} & {\small{}1.08e-13} & 2.50e-04 & {\small{}-6.0e-15}\tabularnewline
\hline 
{\small{}AD-fine} & {\small{}612,309} & {\small{}33,223,067} & \textbf{\small{}1.32e-14} & {\small{}1.84e-07} & 4.63e-05 & -\tabularnewline
\hline 
{\small{}LE-coarse} & {\small{}15,147} & {\small{}610,929} & \textbf{\small{}1.04e-11} & {\small{}6.32e-07} & 3.38e-06 & {\small{}-}\tabularnewline
\hline 
{\small{}LE-mid} & {\small{}109,395} & {\small{}4,652,505} & \textbf{\small{}1.24e-13} & {\small{}6.37e-06} & 5.42e-07 & {\small{}-}\tabularnewline
\hline 
{\small{}LE-fine} & {\small{}1,081,188} & {\small{}47,392,074} & \textbf{\small{}1.53e-11} & {\small{}2.09e-05} & 5.83e-08 & {\small{}-}\tabularnewline
\hline 
\end{tabular}{\small\par}
\end{table}

\section{Conclusions\label{sec:Conclusions}}

In this work, we introduced a new class of variable preconditioners,
called \emph{HIFIR}, for preconditioning asymmetric singular systems.
The core component of HIFIR is a \emph{hybrid incomplete factorization}
or \emph{HIF}, which combines an MLILU with RRQR on the final Schur
complement. We derived HIFIR by establishing the theory that (1) the
generalized inverses are optimal preconditioners for consistent systems
and (2) HIF is an $\epsilon$-accurate approximation to generalized
inverses. We also introduced HIF with iterative refinement to improve
the effectiveness of the preconditioner without tightening the thresholds
in HIF or increasing memory requirements. Using HIF and HIFIR, we
solved three important subclasses of singular systems. First, we applied
HIF in restarted GMRES to find LS solutions of consistent systems,
which significantly improved the accuracy and robustness of MLILU-preconditioned
GMRES. Second, we applied HIFIR in the context of restarted FGMRES
to compute null-space vectors of singular systems. Our proposed approach
improved accuracy by orders of magnitude compared to the prior state
of the art (including sparse SVD and RIF-preconditioned LSMR) while
being significantly faster. Third, by combining the two techniques,
we introduced PIPIT for finding the pseudoinverse solutions of singular
systems from PDEs with low-dimensional null spaces. PIPIT improved
the accuracy by orders of magnitude compared to other alternatives.
Since our proposed methods work for structurally singular systems
(e.g., \textsf{coater2} in Table~\ref{tab:Examples-consistent-systems}),
they can also be applied to solve rank-deficient rectangular LS systems
with $m\approx n$. This work has focused on real matrices, but all
the results generalize to complex matrices. The implementation of
HIFIR for both real and complex matrices is available at \url{https://github.com/hifirworks/hifir}.
However, a limitation of the present work is that PIPIT is not efficient
for solving systems with high-dimensional null spaces, at least in
a sequential setting. Another limitation is that the final Schur complement
in the HIF may be quite large for some systems. In addition, our current
implementation of HIFIR is serial, and it is desirable to parallelize
it for larger-scale problems. We plan to address these limitations
in the future.

\section*{Acknowledgments}

Computational results were obtained using the Seawulf cluster at the
Institute for Advanced Computational Science of Stony Brook University.
We thank the anonymous reviewers for many helpful comments, which
have significantly improved the paper.

\bibliographystyle{siamplain}
\bibliography{references}

\appendix

\section{\label{sec:Relationship-of-condition}Bounding the condition number
of $\boldsymbol{A}\boldsymbol{A}^{g}$}

We assert the following fact.
\begin{proposition}
\label{prop:condition-number}Given $\boldsymbol{A}\in\mathbb{R}^{n\times n}$,
$r=\text{\emph{rank}}(\boldsymbol{A})$, and $\boldsymbol{A}\boldsymbol{A}^{g}=\boldsymbol{X}\begin{bmatrix}\boldsymbol{I}_{r}\\
 & \boldsymbol{0}
\end{bmatrix}\boldsymbol{X}^{-1}$, $\kappa(\boldsymbol{A}\boldsymbol{A}^{g})=\sigma_{1}(\boldsymbol{A}\boldsymbol{A}^{g})/\sigma_{r}(\boldsymbol{A}\boldsymbol{A}^{g})\leq\kappa(\boldsymbol{X})$.
\end{proposition}

Since $\sigma_{r}$ is an interior (instead of extreme) singular value
in general, we cannot simply rely on the Cauchy--Schwarz inequality
to prove the proposition. Instead, we use a hybrid of algebraic and
geometric arguments using both Schur and SVDs.
\begin{proof}
$\sigma_{1}(\boldsymbol{A}\boldsymbol{A}^{g})=\left\Vert \boldsymbol{A}\boldsymbol{A}^{g}\right\Vert \leq\sigma_{1}(\boldsymbol{X})/\sigma_{n}(\boldsymbol{X})=\kappa(\boldsymbol{X})$.
Hence, we only need to show that $\sigma_{r}(\boldsymbol{A}\boldsymbol{A}^{g})\ge1$.
Consider a (real) Schur decomposition \cite[p. 276]{Golub13MC} 
\[
\boldsymbol{A}\boldsymbol{A}^{g}=\boldsymbol{Q}\left[\begin{array}{cc}
\boldsymbol{R}_{11} & \boldsymbol{R}_{12}\\
 & \boldsymbol{R}_{22}
\end{array}\right]\boldsymbol{Q}^{T},
\]
where $\boldsymbol{R}_{11}\in\mathbb{R}^{r\times r}$ and $\boldsymbol{R}_{22}\in\mathbb{R}^{(n-r)\times(n-r)}$
are upper triangular matrices with all ones and zeros in their diagonals,
respectively, and $\boldsymbol{Q}\in\mathbb{R}^{n\times n}$ is orthogonal
(i.e., $\boldsymbol{Q}^{T}=\boldsymbol{Q}^{-1}$). $\boldsymbol{R}=\boldsymbol{Q}^{T}\boldsymbol{A}\boldsymbol{A}^{g}\boldsymbol{Q}$
is idempotent, so are $\boldsymbol{R}_{11}$ and $\boldsymbol{R}_{22}$.
$\boldsymbol{R}_{11}$ is nonsingular and idempotent, so $\boldsymbol{R}_{11}=\boldsymbol{I}_{r}$.
The idempotence of $\boldsymbol{R}_{22}$ implies that $\text{\ensuremath{\mathcal{R}}}(\boldsymbol{R}_{22})=\{\boldsymbol{0}\}$,
so $\boldsymbol{R}_{22}=\boldsymbol{0}$. Therefore, 
\[
\boldsymbol{R}=\boldsymbol{Q}^{T}\boldsymbol{A}\boldsymbol{A}^{g}\boldsymbol{Q}=\left[\begin{array}{cc}
\boldsymbol{I}_{r} & \boldsymbol{R}_{12}\\
 & \boldsymbol{0}
\end{array}\right],
\]
which maps $\pm\boldsymbol{e}_{i}$ to $\pm\boldsymbol{e}_{i}$ for
$i=1,2,\dots,r$. Hence, the hyperellipse \cite[p. 26]{trefethen1997numerical}
$\mathcal{S}=\{\boldsymbol{R}\boldsymbol{u}\mid\boldsymbol{u}\in\mathbb{R}^{n}\wedge\Vert\boldsymbol{u}\Vert=1\}$
encloses the vectors $\{\pm\boldsymbol{e}_{i}\mid i=1,2,\dots,r\}$,
and it must also enclose the $r$-dimensional unit ball with these
axes. Hence, $1\leq\sigma_{r}(\boldsymbol{R})=\sigma_{r}(\boldsymbol{A}\boldsymbol{A}^{g})$.
Due to Definition~\ref{def:condition-number}, $\text{\ensuremath{\kappa}(\ensuremath{\boldsymbol{A}\boldsymbol{A}^{g}})=\ensuremath{\sigma_{1}}(\ensuremath{\boldsymbol{A}\boldsymbol{A}^{g}})/\ensuremath{\sigma_{r}}(\ensuremath{\boldsymbol{A}\boldsymbol{A}^{g}})\ensuremath{\leq\kappa}(\ensuremath{\boldsymbol{X}})}$.
\end{proof}

\section{\label{sec:Proof-of-bound}Proof of bound of $\kappa(\boldsymbol{A}\boldsymbol{G})$}

We prove the approximate bound of $\kappa(\boldsymbol{A}\boldsymbol{G})$
in \eqref{eq:bound-condition-number} in 2-norm. Due to the Cauchy--Schwarz
inequality and Definition~\ref{def:epsilon-accurate},
\[
\left\Vert \boldsymbol{X}^{-1}\boldsymbol{A}\boldsymbol{G}\boldsymbol{X}\right\Vert \leq\left\Vert \begin{bmatrix}\boldsymbol{I}_{r}\\
 & \boldsymbol{0}
\end{bmatrix}\right\Vert +\left\Vert \boldsymbol{X}^{-1}\boldsymbol{A}\boldsymbol{G}\boldsymbol{X}-\begin{bmatrix}\boldsymbol{I}_{r}\\
 & \boldsymbol{0}
\end{bmatrix}\right\Vert \leq1+\epsilon,
\]
and hence 
\[
\left\Vert \boldsymbol{A}\boldsymbol{G}\right\Vert \leq\Vert\boldsymbol{X}\Vert\left\Vert \boldsymbol{X}^{-1}\boldsymbol{A}\boldsymbol{G}\boldsymbol{X}\right\Vert \Vert\boldsymbol{X}^{-1}\Vert=\kappa(\boldsymbol{X})\text{\ensuremath{\left\Vert \boldsymbol{X}^{-1}\boldsymbol{A}\boldsymbol{G}\boldsymbol{X}\right\Vert }}\leq\kappa(\boldsymbol{X})(1+\epsilon).
\]
In other words, for all $\boldsymbol{u}\in\mathcal{R}\left(\boldsymbol{A}\boldsymbol{G}\right)\backslash\{\boldsymbol{0}\}$,
\[
\left\Vert \boldsymbol{u}^{T}\boldsymbol{A}\boldsymbol{G}\right\Vert \leq\left\Vert \boldsymbol{u}\right\Vert \left\Vert \boldsymbol{A}\boldsymbol{G}\right\Vert \leq\kappa(\boldsymbol{X})\left(1+\epsilon\right)\left\Vert \boldsymbol{u}\right\Vert ,
\]
where $\mathcal{R}\left(\boldsymbol{A}\boldsymbol{G}\right)=\mathcal{R}\left(\boldsymbol{A}\right)$
due to Lemma~\ref{lem:AGI-range-preservation}. Furthermore, due
to the Cauchy--Schwarz inequality and submultiplicativity,
\begin{align}
\left\Vert \boldsymbol{u^{T}}\boldsymbol{A}\boldsymbol{G}\right\Vert  & \geq\left\Vert \boldsymbol{u}^{T}\boldsymbol{X}\begin{bmatrix}\boldsymbol{I}_{r}\\
 & \boldsymbol{0}
\end{bmatrix}\boldsymbol{X}^{-1}\right\Vert -\left\Vert \boldsymbol{u^{T}}\boldsymbol{X}\left(\boldsymbol{X}^{-1}\boldsymbol{A}\boldsymbol{G}\boldsymbol{X}-\begin{bmatrix}\boldsymbol{I}_{r}\\
 & \boldsymbol{0}
\end{bmatrix}\right)\boldsymbol{X}^{-1}\right\Vert \nonumber \\
 & \geq\left\Vert \boldsymbol{u}^{T}\boldsymbol{X}\begin{bmatrix}\boldsymbol{I}_{r}\\
 & \boldsymbol{0}
\end{bmatrix}\boldsymbol{X}^{-1}\right\Vert -\kappa(\boldsymbol{X})\epsilon\Vert\boldsymbol{u}\Vert.\label{eq:UTAG_bounds}
\end{align}
Given $\boldsymbol{X}\in\mathbb{R}^{n\times n}$ in Definition~\ref{def:epsilon-accurate},
due to Proposition~\ref{prop:diagnalizability}, there exists $\boldsymbol{A}^{g}$
such that $\boldsymbol{X}^{-1}\boldsymbol{A}\boldsymbol{A}^{g}\boldsymbol{X}=\begin{bmatrix}\boldsymbol{I}_{r}\\
 & \boldsymbol{0}
\end{bmatrix}$ and $\mathcal{R}(\boldsymbol{A}\boldsymbol{A}^{g})=\mathcal{R}(\boldsymbol{A})$.
Hence, 
\begin{equation}
\left\Vert \boldsymbol{u}^{T}\boldsymbol{X}\begin{bmatrix}\boldsymbol{I}_{r}\\
 & \boldsymbol{0}
\end{bmatrix}\boldsymbol{X}^{-1}\right\Vert =\left\Vert \boldsymbol{u}^{T}\boldsymbol{A}\boldsymbol{A}^{g}\right\Vert .\label{eq:uAAg}
\end{equation}
From \eqref{eq:rth-singular-value}, \eqref{eq:UTAG_bounds}, and
\eqref{eq:uAAg},
\begin{align*}
\sigma_{r}(\boldsymbol{A}\boldsymbol{G}) & =\min_{\boldsymbol{u}\in\mathcal{R}(\boldsymbol{A}\boldsymbol{G})\backslash\{\boldsymbol{0}\}}\frac{\left\Vert \boldsymbol{u}^{T}\boldsymbol{A}\boldsymbol{G}\right\Vert }{\left\Vert \boldsymbol{u}\right\Vert }\\
 & \geq\min_{\boldsymbol{u}\in\mathcal{R}(\boldsymbol{A}\boldsymbol{G})\backslash\{\boldsymbol{0}\}}\frac{\left\Vert \boldsymbol{u}^{T}\boldsymbol{X}\begin{bmatrix}\boldsymbol{I}_{r}\\
 & \boldsymbol{0}
\end{bmatrix}\boldsymbol{X}^{-1}\right\Vert }{\left\Vert \boldsymbol{u}\right\Vert }-\kappa(\boldsymbol{X})\epsilon\\
 & =\min_{\boldsymbol{u}\in\mathcal{R}(\boldsymbol{A}\boldsymbol{A}^{g})\backslash\{\boldsymbol{0}\}}\frac{\left\Vert \boldsymbol{u}^{T}\boldsymbol{A}\boldsymbol{A}^{g}\right\Vert }{\left\Vert \boldsymbol{u}\right\Vert }-\kappa(\boldsymbol{X})\epsilon\\
 & =\sigma_{r}(\boldsymbol{A}\boldsymbol{A}^{g})-\kappa(\boldsymbol{X})\epsilon.
\end{align*}
Under the assumption of $\kappa(\boldsymbol{X})\approx1$, $\sigma_{r}(\boldsymbol{A}\boldsymbol{A}^{g})\geq1$
due to Proposition~\ref{prop:condition-number}, and
\[
\kappa(\boldsymbol{A}\boldsymbol{G})=\frac{\sigma_{1}(\boldsymbol{A}\boldsymbol{G})}{\sigma_{r}(\boldsymbol{A}\boldsymbol{G})}\leq\frac{\kappa(\boldsymbol{X})\left(1+\epsilon\right)}{\sigma_{r}(\boldsymbol{A}\boldsymbol{A}^{g})-\kappa(\boldsymbol{X})\epsilon}\approx\frac{1+\epsilon}{1-\epsilon}.
\]

\section{\label{sec:Convergence-iterative-refinement}Convergence of iterative
refinement}

We analyze the convergence properties of iterative refinement over
AGI.
\begin{proposition}
\label{prop:iterative-refinement} If $\mathcal{R}(\boldsymbol{A})\cap\mathcal{N}(\boldsymbol{A})=\{\boldsymbol{0}\}$
and $\mathcal{R}(\boldsymbol{G})=\mathcal{R}(\boldsymbol{A})$, then
the iterative refinement converges for all $\boldsymbol{b}\in\mathcal{R}(\boldsymbol{A})$
and $\boldsymbol{x}_{0}\in\mathcal{R}(\boldsymbol{A})$  if and only
if $\boldsymbol{Q}^{T}(\boldsymbol{I}-\boldsymbol{G}\boldsymbol{A})^{j}\boldsymbol{Q}$
tends to $\boldsymbol{0}$ as $j$ increases, where $\boldsymbol{Q}$
is composed of an orthonormal basis of $\mathcal{R}(\boldsymbol{A})$.
\end{proposition}

\begin{proof}
Under the assumption of $\mathcal{R}(\boldsymbol{G})=\mathcal{R}(\boldsymbol{A})$
and $\boldsymbol{x}_{0}\in\mathcal{R}(\boldsymbol{A})$, $\boldsymbol{x}_{j}\in\mathcal{R}(\boldsymbol{A})$
for all $j\geq0$. Let $\boldsymbol{b}=\boldsymbol{A}\boldsymbol{x}$
for some $\boldsymbol{x}\in\mathbb{R}^{n}$. Since $\mathcal{R}(\boldsymbol{A})\cap\mathcal{N}(\boldsymbol{A})=\{\boldsymbol{0}\}$
is equivalent to $\mathcal{R}(\boldsymbol{A})\oplus\mathcal{N}(\boldsymbol{A})=\mathbb{R}^{n}$,
there exist $\boldsymbol{x}_{*}\in\mathcal{R}(\boldsymbol{A})$ and
$\boldsymbol{x}_{N}\in\mathcal{N}(\boldsymbol{A})$ such that $\boldsymbol{x}=\boldsymbol{x}_{*}+\boldsymbol{x}_{N}$.
Then, $\boldsymbol{x}_{*}$ is an LS solution in that $\boldsymbol{A}\boldsymbol{x}_{*}=\boldsymbol{A}(\boldsymbol{x}-\boldsymbol{x}_{N})=\boldsymbol{b}$.
At the $(j+1\text{st})$ step, $\boldsymbol{x}_{j+1}-\boldsymbol{x}_{*}=(\boldsymbol{I}-\boldsymbol{G}\boldsymbol{A})\boldsymbol{x}_{j}+\boldsymbol{G}\boldsymbol{b}-\boldsymbol{x}_{*}=(\boldsymbol{I}-\boldsymbol{G}\boldsymbol{A})\boldsymbol{x}_{j}+\boldsymbol{G}\boldsymbol{A}\boldsymbol{x}_{*}-\boldsymbol{x}_{*}=(\boldsymbol{I}-\boldsymbol{G}\boldsymbol{A})(\boldsymbol{x}_{j}-\boldsymbol{x}_{*}).$
Let $\boldsymbol{x}_{j}-\boldsymbol{x}_{*}=\boldsymbol{Q}\boldsymbol{s}_{j}$.
Then, $\boldsymbol{Q}\boldsymbol{s}_{j}=(\boldsymbol{I}-\boldsymbol{G}\boldsymbol{A})^{j}\boldsymbol{Q}\boldsymbol{s}_{0}$.
By left-multiplying $\boldsymbol{Q}^{T}$ and using the fact that
$\boldsymbol{Q}^{T}\boldsymbol{Q}=\boldsymbol{I}_{r}$, we have $\boldsymbol{s}_{j}=\boldsymbol{Q}^{T}(\boldsymbol{I}-\boldsymbol{G}\boldsymbol{A})^{j}\boldsymbol{Q}\boldsymbol{s}_{0}$.
Hence, $\left\Vert \boldsymbol{x}_{j}-\boldsymbol{x}_{*}\right\Vert =\left\Vert \boldsymbol{Q}\boldsymbol{s}_{j}\right\Vert =\left\Vert \boldsymbol{s}_{j}\right\Vert \leq\left\Vert \boldsymbol{Q}^{T}(\boldsymbol{I}-\boldsymbol{G}\boldsymbol{A})^{j}\boldsymbol{Q}\right\Vert \left\Vert \boldsymbol{s}_{0}\right\Vert $,
so $\left\Vert \boldsymbol{x}_{j}-\boldsymbol{x}_{*}\right\Vert $
tends to $0$ as $j$ approaches $\infty$ if and only if $\boldsymbol{Q}^{T}(\boldsymbol{I}-\boldsymbol{G}\boldsymbol{A})^{j}\boldsymbol{Q}$
tends to $\boldsymbol{0}$. 
\end{proof}

\end{document}